\documentclass[11pt,twoside]{amsart}

\usepackage{supertabular}

\usepackage[all]{xy}
\usepackage{amsmath}  
\usepackage{mathtools}
\usepackage{amsfonts}
\usepackage{amssymb}
\usepackage{amsthm}
\usepackage{enumitem}
\usepackage{hyperref}
\usepackage{fullpage}
\usepackage{graphicx}
\usepackage{url}

\setlist[enumerate,1]{label=(\roman*)}
\setlist[enumerate,2]{label=(\alph*)}
\setlist[enumerate,3]{label=(\Roman*)}
\setlist[enumerate,4]{label=(\Alph*)}

\theoremstyle{definition}
\newtheorem{defn}{Definition}[section]

\newtheorem{rmk}[defn]{Remark}

\theoremstyle{plain}
\newtheorem{thm}[defn]{Theorem}

\newtheorem{lem}[defn]{Lemma}
\newtheorem{prop}[defn]{Proposition}
\newtheorem{cor}[defn]{Corollary}

\def\C{\ensuremath{\mathbb{C}}}
\def\D{\ensuremath{\mathbb{D}}}

\def\P{\ensuremath{\mathbb{P}}}

\def\R{\ensuremath{\mathbb{R}}}
\def\Z{\ensuremath{\mathbb{Z}}}

\def\AA{\ensuremath{\mathcal A}}
\def\BB{\ensuremath{\mathcal B}}

\def\EE{\ensuremath{\mathcal E}}
\def\FF{\ensuremath{\mathcal F}}

\def\HH{\ensuremath{\mathcal H}}
\def\II{\ensuremath{\mathcal I}}

\def\KK{\ensuremath{\mathcal K}}

\def\NN{\ensuremath{\mathcal N}}
\def\OO{\ensuremath{\mathcal O}}

\def\TT{\ensuremath{\mathcal T}}

\def\Bl{\mathop{\mathrm{Bl}}\nolimits} 

\def\ch{\mathop{\mathrm{ch}}\nolimits}
\def\CH{\mathop{\mathrm{CH}}\nolimits}

\def\Cl{\mathop{\mathrm{Cl}}\nolimits}
\def\Coh{\mathop{\mathrm{Coh}}\nolimits}

\def\Db{\mathop{\mathrm{D}^{\mathrm{b}}}\nolimits}
\def\Ku{\mathop{\mathrm{Ku}}\nolimits}

\def\deg{\mathop{\mathrm{deg}}\nolimits}

\def\dim{\mathop{\mathrm{dim}}\nolimits}

\def\ext{\mathop{\mathrm{ext}}\nolimits}
\def\Ext{\mathop{\mathrm{Ext}}\nolimits}
\def\lExt{\mathop{\mathcal Ext}\nolimits}

\def\gcd{\mathop{{\mathrm{gcd}}}\nolimits}

\def\hom{\mathop{\mathrm{hom}}\nolimits}
\def\Hom{\mathop{\mathrm{Hom}}\nolimits}
\def\lHom{\mathop{\mathcal Hom}\nolimits}

\def\RHom{\mathop{\mathbf{R}\mathrm{Hom}}\nolimits}

\def\Ku{\mathop{\mathrm{Ku}}\nolimits}

\def\Pic{\mathop{\mathrm{Pic}}\nolimits}

\def\Spec{\mathop{\mathrm{Spec}}}

\def\td{\mathop{\mathrm{td}}\nolimits}

\def\into{\ensuremath{\hookrightarrow}}
\def\onto{\ensuremath{\twoheadrightarrow}}

\begin{document}

\title{The desingularization of the theta divisor of a cubic threefold as a moduli space}

\author{Arend Bayer}
\address{School of Mathematics and Maxwell Institute,
University of Edinburgh,
James Clerk Maxwell Building,
Peter Guthrie Tait Road, Edinburgh, EH9 3FD,
United Kingdom}
\email{arend.bayer@ed.ac.uk}

\author{Sjoerd Viktor Beentjes}
\address{School of Mathematics and Maxwell Institute,
University of Edinburgh,
James Clerk Maxwell Building,
Peter Guthrie Tait Road, Edinburgh, EH9 3FD,
United Kingdom}
\email{Sjoerd.Beentjes@ed.ac.uk}

\author{Soheyla Feyzbakhsh}
\address{Department of Mathematics, Imperial College, London SW7 2AZ, United Kingdom}
\email{s.feyzbakhsh@imperial.ac.uk}

\author{Georg Hein}
\address{Universit\"at Duisburg-Essen, Fakult\"at f\"ur Mathematik, Georg Hein, WSC-O-3.58, 45117 Essen, Germany}
\email{georg.hein@uni-due.de}

\author{Diletta Martinelli}
\address{Korteweg-de Vries, Institute for Mathematics, Universiteit van Amsterdam, Science Park 107, 1098XG Amsterdam, Netherlands. }
\email{d.martinelli@uva.nl}

\author{Fatemeh Rezaee}
\address{School of Mathematics and Maxwell Institute,
University of Edinburgh,
James Clerk Maxwell Building,
Peter Guthrie Tait Road, Edinburgh, EH9 3FD,
United Kingdom}
\email{fatemeh.rezaee@ed.ac.uk}

\author{Benjamin Schmidt}
\address{Gottfried Wilhelm Leibniz Universit\"at Hannover, Institut f\"ur Algebraische Geometrie, Welfengarten 1, 30167 Hannover, Germany}
\email{bschmidt@math.uni-hannover.de}



\begin{abstract}
We show that the moduli space $\overline{M}_X(v)$ of Gieseker stable sheaves on a smooth cubic threefold $X$ with Chern character $v = (3,-H,-H^2/2,H^3/6)$ is smooth and of dimension four. Moreover, the Abel-Jacobi map to the intermediate Jacobian of $X$ maps it birationally onto the theta divisor $\Theta$, contracting only a copy of $X \subset \overline{M}_X(v)$ to the singular point $0 \in \Theta$. 

We use this result to give a new proof of a categorical version of the Torelli theorem for cubic threefolds, which says that $X$ can be recovered from its Kuznetsov component $\Ku(X) \subset \Db(X)$. Similarly, this leads to a new proof of the description of the singularity of the theta divisor, and thus of the classical Torelli theorem for cubic threefolds, i.e., that $X$ can be recovered from its intermediate Jacobian.
\end{abstract}

\maketitle

\tableofcontents

\section{Introduction} 

Moduli spaces of sheaves provide examples of algebraic varieties with an interesting and rich geometry and they have been widely studied in the past few decades. In particular, there are many strong results regarding moduli spaces on surfaces, while the situation on threefolds is  less understood. We refer to \cite{HL10:moduli_sheaves} for a more detailed account of the theory, which has  been revolutionized by the introduction of stability conditions on triangulated categories by Bridgeland \cite{Bri07:stability_conditions}. 

Perhaps the main player of the seminal paper by Clemens and Griffiths \cite{CG72:cubic_threefolds} on the geometry of cubic threefolds is the theta divisor $\Theta$ of its intermediate Jacobian $J(X)$. Various authors have studied parametrizations of the theta divisor by moduli spaces of sheaves, see \cite{AKP04:thetamodel, Bea02:vector_bundles_cubic_threefold, Iliev:conicbundlesections}.

In this paper, we find a new, and in a sense most efficient, parametrization of this type: a smooth four-dimensional moduli space of stable sheaves isomorphic to the desingularization of the theta divisor. 

Let $X \subset \P^4$ be a smooth cubic threefold over $\C$ and $H$ the hyperplane section. Let $\overline{M}_X(v)$ be the moduli space of Gieseker-semistable sheaves on $X$ with Chern character $v \coloneqq (3, -H, -\tfrac{1}{2} H^2, \tfrac{1}{6} H^3)$. 

\newtheorem*{thm:main}{Theorem \ref{thm:main}}
\begin{thm:main}
The moduli space $\overline{M}_X(v)$ is smooth and irreducible of dimension $4$. More precisely, it is the blow up of $\Theta$ in its unique singular point. The exceptional divisor is isomorphic to the cubic threefold $X$ itself, and parametrizes non-locally free sheaves in $\overline{M}_X(v)$. 
\end{thm:main}

\subsection*{Moduli space in the Kuznetsov component}
The original motivation for our analysis of the moduli space $\overline{M}_X(v)$ comes from the study of moduli spaces of stable objects in a full triangulated subcategory $\Ku(X) \subset \Db(X)$ called the \emph{Kuznetsov component}. It is defined through the semi-orthogonal decomposition
\[
\Db(X) = \langle \Ku(X), \OO_X, \OO_X(H) \rangle.
\]
See \cite{Kuz04:SOD} for details on the decomposition and on the Kuznetsov component. 

Stability conditions on $\Ku(X)$ have been constructed in \cite{BMMS12:cubics} and \cite{BLMS17:stability_kuznetsov}. These stability conditions are Serre-invariant, which roughly means that stability of an object is preserved by the action of the Serre functor of $\Ku(X)$ (see Section \ref{sec:kuznetsov} for the precise definition). This property allows us to study stability of objects irrespective of the specific construction of stability conditions.


The class $v$ in Theorem \ref{thm:main} is chosen as the class of the projection $K_P$ of a skyscraper sheaf $\OO_P$ for a point $P \in X$, which is defined by the short exact sequence
\[ 0 \to K_P \to \OO^{\oplus 4} \to I_P(1) \to 0. \]
These are the non-locally free torsion-free slope-stable sheaves appearing in Theorem~\ref{thm:main}, and we show that they are also stable as objects of $\Ku(X)$ with respect to any Serre-invariant stability condition. Hence, the moduli space $M_\sigma(v)$ of $\sigma$-stable objects in $\Ku(X)$ of Chern character $v$ contains $X$, yet its expected dimension is four. This was our first clue that this moduli space is of interest. Indeed, our next result says that the moduli spaces $M_\sigma(v)$ and $\overline{M}_X(v)$ agree entirely.

\begin{thm}[Theorem~\ref{thm:equalmodspace} and Proposition~\ref{prop:unique}]
Let $\sigma$ be an arbitrary Serre-invariant stability condition on $\Ku(X)$. Then the moduli space $M_{\sigma}(v)$ is isomorphic to the moduli space $\overline{M}_X(v)$.
\end{thm}


To summarise, we project the structure sheaf of a point into the Kuznetsov component and take its moduli space. It obviously contains  $X$ but is bigger. It is the resolution of the theta divisor, with $X$ as the exceptional divisor. Thus, we recover $X$ from $\Ku(X)$ or from the intermediate Jacobian, i.e., we obtain new proofs of both the categorical and classical Torelli theorem for cubic threefolds:  

\begin{thm}[Corollary \ref{cor:classicalTorelli} and Theorem \ref{thm:categorical-Torelli}]
\label{thm:Torelli}
Let $X_1$ and $X_2$ be smooth cubic threefolds. The following are equivalent:
\begin{enumerate}
    \item\label{item:isom} $X_1$ and $X_2$ are isomorphic.
    \item\label{item:Kuequiv} $\Ku(X_1)$ and $\Ku(X_2)$ are equivalent as triangulated categories.
    \item \label{item:JXisom} $J(X_1)$ and $J(X_2)$ are isomorphic as principally polarised abelian varieties.
\end{enumerate}
\end{thm}

\subsection*{Proof ideas}

The proof of Theorem \ref{thm:main} relies on two classical ingredients. Firstly, we use the fact that any irreducible theta divisor is normal due to \cite{EL97:singularities_theta_divisors}. Secondly, we use a characterization of the theta divisor of the intermediate Jacobian in terms of twisted cubics, see Proposition \ref{prop:abel-jacobi_twisted_cubic}. This was proved by Beauville in \cite{Bea02:vector_bundles_cubic_threefold}, but it can also be deduced from the description of $\Theta$ as differences of lines in \cite{CG72:cubic_threefolds}, see Remark \ref{rmk:abel-jacobi_twisted_cubic}.

The strategy to prove Theorem \ref{thm:main} is to vary the notion of stability and reach a detailed description of the objects that belong to the moduli space $\overline{M}_X(v)$ through wall-crossing. Since $X$ has Picard rank one, Gieseker stability cannot be varied. This is where the derived category comes into play in the form of tilt-stability introduced in \cite{Bri08:stability_k3} for K3 surfaces, and then further generalized to other surfaces and threefolds in \cite{AB13:k_trivial, BMT14:stability_threefolds}. In fact, we give a complete description of the wall and chamber structure (see Section \ref{sec:wallstructure}). 
Once a set-theoretic description of $\overline{M}_X(v)$ has been reached, we use standard deformation theory arguments in Corollary \ref{cor:smooth_4dim} to deduce that it is smooth and of dimension four.

To prove Theorem \ref{thm:equalmodspace}, we first show the claim for the specific stability condition constructed in \cite{BLMS17:stability_kuznetsov} which are Serre-invariant by \cite{PY20:Fano3index2}. We then prove in a completely separate argument that our moduli space is independent of the choice of Serre-invariant stability conditions $\sigma$. The essential ingredient in this last argument is the weak Mukai Lemma from \cite{PY20:Fano3index2}.  

\subsection*{Related work}
In the recent paper \cite{APR19:kuznetsov_fano_torsion_moduli}, the authors studied moduli spaces of some torsion sheaves in the Kuznetsov components of Fano threefolds with Picard rank one and index two. In the case of cubic threefolds they study $M_{\sigma}([S^2(K_P)])$ ($S$ is the Serre functor on $\Ku(X)$), but do not obtain our detailed geometric description. A key difference is that in their case the moduli space in the Kuznetsov component is different from the moduli space of Gieseker-semistable sheaves.

Classical Torelli is the implication \ref{item:JXisom} $\Rightarrow$ \ref{item:isom} in Theorem \ref{thm:Torelli} which was first proved in \cite{CG72:cubic_threefolds}. The implication \ref{item:Kuequiv} $\Rightarrow$ \ref{item:JXisom} was first established in \cite[Theorem 1.1]{BMMS12:cubics}, where it was shown that the Fano variety of lines $F(X)$ can be recovered from $\Ku(X)$ as a moduli space of stable objects. Thus, one obtains the intermediate Jacobian $J(X)$ as the Albanese variety of $F(X)$. A more recent argument for \ref{item:Kuequiv} $\Rightarrow$ \ref{item:JXisom} can be deduced from Perry's categorical construction of intermediate Jacobians \cite[Section 5.3]{Perry:integralHodge}, when the equivalence is given by a Fourier-Mukai kernel on $X_1 \times X_2$. Instead, our paper gives a very direct geometric argument for \ref{item:Kuequiv} $\Rightarrow$ \ref{item:isom}, as well as a variant of the proof of classical Torelli via the description of the singularity of theta divisor implied by Theorem \ref{thm:main}.

Since this article originally appeared on the arXiv, \cite{FP} and \cite{Zhang} proved uniqueness of Serre-invariant stability conditions on $\Ku(X)$. Proposition \ref{prop:unique} in the last section could now be obtained as an immediate corollary.


\subsection*{Acknowledgements}
This work originated from a working group at the workshop ``Semiorthogonal decompositions, stability conditions and sheaves of categories" in Toulouse in May 2018, a meeting of the CIMI network ``Semiorthogonal decompositions and stability condition''. We are deeply grateful to the organizers of the workshop for the nice and productive atmosphere. In particular, we would like to thank Aaron Bertram, Naoki Koseki, Marin Petkovic, Maxim Smirnov, and Benjamin Sung for very useful discussions we had in Toulouse at the beginning of the project. We would also like to express our gratitude to Marcello Bernardara, Emanuele Macr{\`i}, Alex Perry, Laura Pertusi, and Richard Thomas for many useful comments on the subject of this paper. In addition, we are very grateful to the referee for many detailed comments, making the exposition more accurate and readable.

A.~B., S.~B., S.~F., D.~M. and F.~R. were in part supported by ERC starting grant no.~337039. A.~B. and F.~R. were in part supported by ERC consolidator grant no.~819864 WallCrossAG.  F.~R. was partially supported by EPSRC grant EP/T015896/1 and a Maxwell Institute Research Fellowship. S.~F. was partially supported by EPSRC postdoctoral fellowship EP/T018658/1. This material is also partially based upon work supported by the NSF under Grant No. DMS-1440140 while A.~B., D.~M. and F.~R. were in residence at the MSRI in Berkeley, California, during the spring 2019 semester.


\subsection*{Notation}

\begin{center}
   \begin{supertabular}{ r l }
     $X$ & smooth cubic threefold in $\P^4$ over $\C$ \\
     $H$ & the ample generator of $\Pic(X)$ \\
     $Y$ & a hyperplane section of $X$ \\
     $\Db(X)$ & bounded derived category of coherent sheaves on $X$ \\
     $\Ku(X)$ & the Kuznetsov component inside $\Db(X)$ \\
     $\CH^*(X)$ & the Chow ring of $X$ \\ 
     $\CH_{\text{n}}^*(X)$ & the numerical Chow ring of $X$, obtained as $\CH^*(X)$ modulo numerical equivalence \\
     $\HH^{i}(E)$ & the $i$-th cohomology sheaf of a complex $E \in \Db(X)$ \\
     $H^i(E)$ & the $i$-th sheaf cohomology group of a complex $E \in \Db(X)$ \\
     $\ch(E)$ & total Chern character of an object $E \in \Db(X)$ up to numerical equivalence \\
     $c(E)$ & total Chern class of an object $E \in \Db(X)$ up to numerical equivalence \\
     $\widetilde{\ch}(E)$ & total Chern character of an object $E \in \Db(X)$ up to rational equivalence \\
     $\widetilde{c}(E)$ & total Chern class of an object $E \in \Db(X)$ up to rational equivalence \\
     $\ch_{\leq l}(E)$ & $(\ch_0(E), \ldots, \ch_l(E))$ \\
     $\widetilde{\ch}_{\leq l}(E)$ & $(\widetilde{\ch}_0(E), \ldots, \widetilde{\ch}_l(E))$ \\
   \end{supertabular}
\end{center}


\section{Cubic threefolds and intermediate Jacobians} \label{sec:2}

Let $X \subset \P^4$ be a smooth cubic threefold. In their celebrated article \cite{CG72:cubic_threefolds}, Clemens and Griffiths introduced the \emph{intermediate Jacobian} of $X$. It is the complex torus defined as
\[
J(X) \coloneqq H^{2,1}(X)^{\vee}/H_3(X, \Z) = H^1(\Omega^2_X)^{\vee}/H_3(X, \Z).
\]
It turns out that $J(X)$ is a principally polarized abelian variety of dimension five.

Let $\{Z_b\}_{b \in \BB}$ be a family of $1$-cycles over a variety $\BB$. The choice of a base point $b_0 \in \BB$ leads to an Abel-Jacobi map $\Psi_{\BB}: \BB \to J(X)$ as follows. For any $b \in \BB$ the cycle $Z_b - Z_{b_0}$ has degree $0$, i.e., it is homologically trivial, and can be written as the boundary $\partial \Gamma$ for a $3$-chain $\Gamma$. The integral $\int_\Gamma$ is an element in $H^{1,2}(X)^{\vee}$ whose class in $J(X)$ is the image of the Abel-Jacobi map. By \cite[Theorem 2.20]{Gri68:periodsII} the map $\Psi_{\BB}$ is algebraic along the smooth locus of $\BB$.

If $Z_b = C$ is a smooth curve, then the induced morphism on tangent spaces has been described by Welters, see \cite[Section 2]{Wel81:abel_jacobi}. Recall that the tangent space of the Hilbert scheme at $C$ is naturally given by $H^0(\NN_{C/X})$ where $\NN_{C/X}$ is the normal bundle. The tangent space of $J(X)$ at any point is given by $H^{1,2}(X)^{\vee} = H^1(\Omega^2_X)^{\vee}$. By definition, the \emph{infinitesimal Abel-Jacobi map} $\psi_C \colon H^0(\NN_{C/X}) \to H^1(\Omega^2_X)^{\vee}$ is the map of tangent spaces induced by $\Psi_{\BB}$. We get a dual morphism $\psi_C^{\vee}\colon H^1(\Omega^2_X) \to H^0(\NN_{C/X})^{\vee}$.


\begin{lem}
\label{lem:infinitesimal_abel_jacobi}
The following diagram is commutative and has exact rows and columns.

\centerline{
\xymatrix{
& 0 \ar[d] & \\
& H^0(\II_C(H)) \ar[d] & \\
& H^0(\OO_X(H)) \ar[r]^{\cong} \ar[d] & H^1(\Omega^2_X) \ar[d]^{\psi_C^{\vee}} \\
H^0(\NN_{C/\P^4}(-2H)) \ar[r] & H^0(\OO_C(H)) \ar[r] & H^0(\NN_{C/X})^{\vee}
}}
\end{lem}

\begin{proof}
This is mostly \cite[Lemma 2.8]{Wel81:abel_jacobi} and the preceding construction of the morphisms. The map $H^0(\OO_X(H)) \to H^1(\Omega^2_X)$ is the connecting morphism in a long exact sequence
\[
H^0(\Omega^3_{\P^4} \otimes \OO_X(3H)) \to H^0(\OO_X(H)) \to H^1(\Omega^2_X) \to H^1(\Omega^3_{\P^4} \otimes \OO_X(3H)).
\]
The wedge product induces a perfect pairing $\Omega^3_{\P^4} \otimes \Omega_{\P^4} \to \OO_{\P^4}(-5)$. Therefore, $\Omega^3_{\P^4} = T_{\P^4}(-5)$. For $i = 0, 1$ we have
\[
H^i(T_{\P^4} \otimes \OO_X(-2H)) = 0. \qedhere
\]
\end{proof}

Recall that the Lefschetz Hyperplane Theorem says that the hyperplane section $H \in \Pic(X)$ generates the Picard group.
One can use twisted cubics to characterize the theta divisor of $J(X)$. A proof of the following result can be found in \cite[Proposition 5.2]{Bea02:vector_bundles_cubic_threefold}.
Let $\TT$ be the open locus of smooth twisted cubics in the Hilbert scheme of $X$, and let $\overline{\TT}$ be its closure.

\begin{prop}
\label{prop:abel-jacobi_twisted_cubic}
The Abel-Jacobi map $\varphi: \overline{\TT} \to J(X)$ with base point of class $H^2$ is algebraic. Its image is a theta divisor $\Theta \subset J(X)$ and its generic fiber is isomorphic to $\P^2$.
\end{prop}

\begin{rmk}
\label{rmk:abel-jacobi_twisted_cubic}
Proposition \ref{prop:abel-jacobi_twisted_cubic} can be deduced from the description of $\Theta$ as differences of lines as well. We give a rough sketch of the argument here.

Let $F$ be the Fano variety of lines on $X$. According to \cite{CG72:cubic_threefolds} the morphism $F \times F \to J(X)$ that maps $(L, L') \mapsto [L] - [L']$ is generically a $6$ to $1$ cover of $\Theta$.

Since a twisted cubic $C \subset X$ lies in a unique cubic surface $Y \subset X$, the morphism $\TT \to J(X)$ factors via the moduli space $\FF$ of pairs $(D, Y)$, where $Y$ is a cubic surface and $D$ is the divisor class of a twisted cubic. The generic fiber of the morphism $\TT \to \FF$ is given by $\P(H^0(\OO_Y(D)) = \P^2$. Indeed, $\OO_Y(D)$ is the pullback of $\OO_{\P^2}(1)$ if $Y$ is written as the blow up of six general points in $\P^2$.

If $D$ is the class of a twisted cubic on a smooth cubic surface, then $D - H^2$ can be written as the difference of two lines on a cubic surface. Therefore, the Abel-Jacobi morphism maps onto $\Theta$. Moreover, there are precisely six ways to write $D - H^2$ as the difference of two lines. Together with 
 the fact that $F \times F \to J(X)$ is generically a $6$ to $1$ cover of $\Theta$, we get that $\FF \to \Theta$ has degree $1$.
\end{rmk}


\begin{lem}
\label{lem:normal_bundles}
Let $\P^1 \cong C \subset X \subset \P^4$ be a twisted cubic. Then $\NN_{C/\P^4} = \OO_{\P^1}(5)^{\oplus 2} \oplus \OO_{\P^1}(3)$, $h^0(\NN_{C/X}) = 6$, and $h^1(\NN_{C/X}) = 0$. In particular, the Hilbert scheme $\TT$ is smooth of dimension six.
\end{lem}

\begin{proof}
We have a short exact sequence
\[
0 \to \NN_{C/\P^3} = \OO_{\P^1}(5)^{\oplus 2} \to \NN_{C/\P^4} \to \NN_{\P^3/\P^4} \otimes \OO_C = \OO_{\P^1}(3) \to 0.
\]
Since $\Ext^1(\OO_{\P^1}(3), \OO_{\P^1}(5)) = 0$, we get $\NN_{C/\P^4} = \OO_{\P^1}(5)^{\oplus 2} \oplus \OO_{\P^1}(3)$. Next, we have a short exact sequence
\[
0 \to \NN_{C/X} \to \NN_{C/\P^4} = \OO_{\P^1}(5)^{\oplus 2} \oplus \OO_{\P^1}(3) \to \NN_{X/\P^4} \otimes \OO_C = \OO_{\P^1}(9) \to 0. 
\]
Thus, $\NN_{C/X}$ has degree $4$ and can only be $\OO_{\P^1}(m) \oplus \OO_{\P^1}(4 - m)$ for some $-1 \leq m \leq 5$. The claim about the cohomology of $\NN_{C/X}$ holds for each of them.
\end{proof}

\begin{lem}
\label{lem:abel_jacobi_infinitesimal}
Along the locus of smooth curves $\TT \subset \overline{\TT}$,
the Abel-Jacobi morphism $\varphi$ has differential of rank four.
\end{lem}

\begin{proof}
Let $C \subset X$ be a smooth twisted cubic. Clearly, restriction maps $H^0(\OO_X(H)) \cong \C^5$ surjectively onto $H^0(\OO_C(H)) \cong \C^4$. By Lemma \ref{lem:normal_bundles}, we have $h^0(N_{C/\P^4}(-2H)) = h^0(\OO_{\P^1}(-1)^{\oplus 2} \oplus \OO_{\P^1}(-3)) = 0$. By Lemma \ref{lem:infinitesimal_abel_jacobi}, we get a commutative diagram

\centerline{
\xymatrix{
\C^5 \ar[r]^{\cong} \ar@{->>}[d] & H^1(\Omega^2_X) \ar[d]^{\psi_C^{\vee}} \\
\C^4 \ar@{^{(}->}[r] & H^0(\NN_{C/X})^{\vee}
}} 

Therefore, $\psi_C^{\vee}$ has rank four.
\end{proof}

The singularities of the theta divisor were computed in \cite[p. 348]{Mum74:prym_varieties}. Another proof was given in \cite[Main Theorem and Proposition 2]{Bea82:singularities_theta_divisor}. We will not need this full description and instead rely only on normality.

\begin{thm}[{\cite[Theorem 1]{EL97:singularities_theta_divisors}}]
\label{thm:theta_divisor_normal}
Any irreducible theta divisor of an abelian variety is normal.
\end{thm}

\begin{lem}
\label{lem:todd_of_cubic}
Up to numerical equivalence, the Todd class of $X$ is $\td(X) = (1, H, \tfrac{2}{3} H^2, \tfrac{1}{3} H^3)$. In particular, for any $E \in \Db(X)$
\[
\chi(E) = \ch_3(E) + H \cdot \ch_2(E) + \tfrac{2}{3} H^2 \cdot \ch_1(E) + \tfrac{1}{3} H^3 \cdot \ch_0(E).
\]
\end{lem}

\begin{proof}
By Kodaira vanishing $H^i(\OO_X) = 0$ for $i \neq 0$, and therefore, $\chi(\OO_X) = 1$. By the Hirzebruch-Riemann-Roch Theorem we get $\td_{3}(X) = \chi(\OO_X) = \tfrac{1}{3} H^3$. Similarly, Kodaira vanishing implies $H^i(\OO_X(-H)) = 0$ for $i \neq 0$. Again by Hirzebruch-Riemann-Roch
\[
0 = \chi(\OO_X(-H)) = \frac{-H^3}{6} + H \cdot \frac{H^2}{2} - \td_2(X) \cdot H + \frac{H^3}{3}.
\]
Since $X$ has Picard rank one, this is only possible if $\td_2(X) = \tfrac{2}{3} H^2$.
\end{proof}

\begin{lem}
\label{lem:chow_ring}
The numerical Chow ring $\CH_{\text{n}}^*(X)$ has a basis given by $1$, $H$, $H^2/3$, and $H^3/3$. In particular, if $E \in \Db(X)$, then $\ch_2(E) \in \tfrac{1}{6} H^2 \cdot \Z$, and $\ch_3(E) \in \tfrac{1}{6} H^3 \cdot \Z$.
\end{lem}

\begin{proof}
Since $\Pic(X)$ is generated by $H$, the group $\CH_{\text{n}}^2(X)$ is generated by a rational multiple of $H^2$. A general hyperplane section of $X$ is a smooth cubic surface, which contains lines. The class of such a line is $H^2/3$. Since $H^3 = 3$, the class has to be indivisible. Since $H^3/3$ is the class of a point, the group $\CH_{\text{n}}^3(X)$ must be generated by it.

The claim about second Chern characters follows directly from $\ch_2(E) = \tfrac{1}{2} c_1^2(E) - c_2(E)$. The claim about $\ch_3(E)$ follows from Lemma \ref{lem:todd_of_cubic} and the fact that $\chi(E) \in \Z$.
\end{proof}


\section{Divisors on hyperplane sections} \label{sec:3}

We need to understand the singularities that can occur on hyperplane sections of $X$.

\begin{prop}
\label{prop:normal_hyperplane_section}
Any cubic hyperplane section $Y = V \cap X \subset \P^4$ is normal and integral.
\end{prop}

\begin{proof}
Since hypersurfaces satisfy condition S2, by Serre's condition \cite[\href{https://stacks.math.columbia.edu/tag/031S}{Section 031S}]{stacks-project}, it is enough to show that $Y$ has isolated singularities. Assume for contradiction that $Y$ contains a curve $C$ of singular points. Let $F$ and $x$ be the defining equation of $X$ and $V$, respectively. Then $\partial F/\partial x$ is a homogeneous degree 2 polynomial and hence vanishes somewhere along $C$. At such a point, all partial derivatives of $F$ vanish, hence it is a singular point of $X$, a contradiction.
\end{proof}


In order to deal with singular hyperplane sections, we need to recall the relation between Weil divisors and rank one reflexive sheaves on integral normal varieties. This is very similar to the standard relation between line bundles and Cartier divisors. We refer to \cite[\href{https://stacks.math.columbia.edu/tag/0EBK}{Section 0EBK}]{stacks-project} or \cite{Sch07:generalized_divisors} for proofs of the following facts. They can also be found in \cite{Har94:generalized_divisors} in more generality.

Let $Y$ be a normal integral projective variety. By $\Cl(Y)$ we denote the \emph{group of Weil divisors modulo rational equivalence}. For two rank one reflexive sheaves $L_1, L_2 \in \Coh(Y)$ we can define a new rank one reflexive sheaf by $(L_1 \otimes L_2)^{\vee \vee}$. This defines a group law for rank one reflexive sheaves on $Y$, where inverses are given by $L \mapsto L^{\vee}$. For any effective prime divisor $D$ one can define a rank one reflexive sheaf $\OO_Y(D) \coloneqq \II_D^{\vee}$. This can be linearly extended to any divisor.

\begin{prop}
\begin{enumerate}
    \item The group of isomorphism classes of rank one reflexive sheaves is isomorphic to $\Cl(Y)$ under the homomorphism $D \mapsto \OO_Y(D)$.
    \item To every non-zero section $s \in H^0(L)$ of a rank one reflexive sheaf $L$, one can associate an effective divisor $D$ on $Y$.
    \item For any effective Weil-divisor $D$ on $Y$ there is a section $s \in H^0(\OO_Y(D))$ such that the associated divisor is given by $D$.
    \item Two sections $s_1, s_2 \in H^0(L)$ define the same divisor if they satisfy $s_1 = \lambda s_2$ for some $\lambda \in \C^*$.
\end{enumerate}
\end{prop}


\section{Notions of stability}
\label{sec:stability}
In this section, we recall a number of notions of stability for sheaves. Let $X$ be a smooth projective threefold, and let $H$ be an ample divisor on $X$.

\begin{defn}[\cite{Mum63:quotients, Tak72:Stability1}]
\begin{enumerate}
    \item For any $E \in \Coh(X)$, the \emph{Mumford-Takemoto-slope} is defined as
    \[
    \mu(E) \coloneqq \begin{cases}
    \frac{H^2 \cdot \ch_1(E)}{H^3 \cdot \ch_0(E)} & \text{for $\ch_0(E) \neq 0$,} \\
    +\infty & \text{for $\ch_0(E) = 0$.}
    \end{cases}
    \]
    \item A sheaf $E \in \Coh(X)$ is \emph{slope-(semi)stable} if for any non-trivial proper subsheaf $F \into E$ the inequality $\mu(F) < (\leq) \mu(E/F)$ holds.
\end{enumerate}
\end{defn}

From the Definition it follows immediately that if $\Pic(X) = \Z \cdot H$ and $E$ is slope-semistable with $\gcd\left(\ch_0(E), \, H^2\ch_1(E)/H^3\right) =1$, then $E$ is slope-stable. 

While slope-stability suffices to construct moduli spaces of vector bundles on curves, a refinement is necessary in higher dimensions. 

\begin{defn}
We define a pre-order on the polynomial ring $\R[m]$ as follows.
\begin{enumerate}
    \item For all non-zero $f \in \R[m]$, we have $f \prec 0$.
    \item If $\deg(f) > \deg(g)$ for non-zero $f, g \in \R[m]$, then $f \prec g$.
    \item Let $\deg(f) = \deg(g)$ for non-zero $f, g \in \R[m]$ and let $a_f$ and $a_g$ be the leading coefficient of $f$ and $g$. Then $f \preceq g$ if and only if $f(m)/a_f \leq g(m)/a_g$ for all $m \gg 0$.
    \item If $f, g \in \R[m]$ with $f \preceq g$ and $g \preceq f$, we write $f \asymp g$.
\end{enumerate}
\end{defn}

For any $E \in \Coh(X)$, we denote its \emph{Hilbert polynomial} and the terms $\alpha_i(E)$ by $P(E, m) \coloneqq \chi(E(mH)) = \sum_{i = 0}^3 \alpha_i(E) m^i$. Moreover, let $P_2(E, m) = \sum_{i = 1}^3 \alpha_i(E) m^i$.

\begin{defn}
\begin{enumerate}
    \item The sheaf $E$ is Gieseker-(semi)stable if for all non-trivial proper subsheaves $F \subset E$, the inequality $P(F, m) \prec (\preceq) P(E, m)$ holds.
    \item The sheaf $E$ is $2$-Gieseker-(semi)stable if for all non-trivial proper subsheaves $F \subset E$, the inequality $P_2(F, m) \prec (\preceq) P_2(E/F, m)$ holds.
\end{enumerate}
\end{defn}

Note that for $2$-Gieseker-semistability we could have equivalently asked $P_2(F, m) \preceq P_2(E, m)$, but for $2$-Gieseker-stability, $P_2(F, m) \prec P_2(E, m)$ is a stronger condition that is almost never fulfilled for all such subsheaves. These notions imply each other as follows:

\centerline{
\xymatrix{
\text{slope-stable} \ar@{=>}[r] & \text{$2$-Gieseker-stable} \ar@{=>}[r] & \text{Gieseker-stable} \ar@{=>}[d] \\
\text{slope-semistable} & \text{$2$-Gieseker-semistable} \ar@{=>}[l] & \text{Gieseker-semistable} \ar@{=>}[l]
}}

The intermediate notion of $2$-Gieseker stability is not classical and will just appear in the technical parts of our arguments. 

Due to \cite{Gie77:vector_bundles, Mar77:stable_sheavesI, Mar78:stable_sheavesII, Sim94:moduli_representations} there exists a projective moduli space $\overline{M}_X(v)$ parametrising S-equivalence classes of Gieseker-semistable sheaves with Chern character $v$. Here two semistable sheaves are called \emph{S-equivalent} if they have the same stable factors up to order and isomorphism in their Jordan-H\"older filtrations:

\begin{prop}[{\cite[Proposition 1.5.2]{HL10:moduli_sheaves}}]
Any Gieseker-semistable sheaf $E \in \Coh(X)$ has a filtration 
\[
0 = E_0 \into E_1 \into \ldots \into E_n = E
\]
such that the factors $A_i \coloneqq E_i/E_{i-1}$ are Gieseker-stable with $P(A_i, m) \asymp P(E, m)$ for $i = 1, \ldots, n$. The sheaf
\[
\bigoplus_{i = 1}^n A_i
\]
is uniquely determined (up to isomorphism) by $E$.
\end{prop}

Moreover, any sheaf $E$ has a \emph{Harder-Narasimhan-filtration} into Gieseker-semistable factors.

\begin{prop}[{\cite[Theorem 1.3.4]{HL10:moduli_sheaves}}]
Let $E \in \Coh(X)$. There is a unique filtration
\[
0 = E_0 \into E_1 \into \ldots \into E_n = E
\]
such that the factors $A_i \coloneqq E_i/E_{i-1}$ are Gieseker-semistable with $P(A_1, m) \succ P(A_2, m) \succ \ldots \succ P(A_n, m)$.
\end{prop}

Based on Bridgeland stability on surfaces, the notion of tilt stability was introduced in \cite{BMT14:stability_threefolds}. It is not quite a Bridgeland stability condition, but it turns out to suffice for our purposes. The basic idea is to change the category in which subobjects are taken when defining stability. This is done via the theory of tilting introduced in \cite{HRS96:tilting}. As before, let $X$ be a smooth projective threefold with an ample divisor $H$.

\begin{defn}
For any $\beta \in \R$, we define two full additive subcategories of $\Coh(X)$:
\begin{align*}
    \FF_{\beta}(X) &\coloneqq \{ E \in \Coh(X) : \text{any slope-semistable factor $F$ of $E$ satisfies $\mu(F) \leq \beta$}\}, \\
    \TT_{\beta}(X) &\coloneqq \{ E \in \Coh(X) : \text{any slope-semistable factor $F$ of $E$ satisfies $\mu(F) > \beta$}\}.
\end{align*}
The category
\[
\Coh^{\beta}(X) \coloneqq \langle \TT_{\beta}(X), \FF_{\beta}(X)[1] \rangle
\]
is the full additive subcategory of those $E \in \Db(X)$ for which $\HH^0(E) \in \TT_{\beta}(X)$, $\HH^{-1}(E) \in \FF_{\beta}(X)$, and $\HH^i(E) = 0$ for all $i \neq -1, 0$.
\end{defn}
Note that $\Hom(T,F) = 0$ for all $T \in \TT_{\beta}(X)$ and $F \in \FF_{\beta}(X)$, by semistability.
It is well known that the category $\Coh^{\beta}(X)$ is abelian. A sequence of morphisms
\[
0 \to A \to B \to C \to 0
\]
in $\Coh^{\beta}(X)$ is a short exact sequence if and only if the induced sequence
\[
A \to B \to C \to A[1]
\]
is a distinguished triangle in $\Db(X)$.

To simplify notation, we define for any $E \in \Db(X)$ its \emph{twisted Chern character} as $\ch^{\beta}(E) \coloneqq \ch(E) \cdot e^{-\beta H}$. Note that when $\beta \in \Z$, this is nothing but $\ch(E \otimes \OO_X(-\beta H))$.

\begin{defn}
For $\alpha > 0$, $\beta \in \R$, and $E \in \Coh^{\beta}(X)$ we define a slope function
\[
\nu_{\alpha, \beta}(E) \coloneqq \frac{H \cdot \ch_2^{\beta}(E) - \frac{\alpha^2}{2} H^3 \cdot \ch_0^{\beta}(E)}{H^2 \cdot \ch_1^{\beta}(E)},
\]
where again division by zero needs to be interpreted as $+\infty$. Analogously to slope stability, an object $E \in \Coh^{\beta}(X)$ is called \emph{$\nu_{\alpha, \beta}$-(semi)stable} if for all non-trivial proper subobjects $F \into E$ in $\Coh^{\beta}(X)$ the inequality $\nu_{\alpha, \beta}(F) < (\leq) \nu_{\alpha, \beta}(E/F)$ holds.
\end{defn}

If it is clear from context, we will sometimes abuse notation and write \emph{tilt-(semi)stable} instead of $\nu_{\alpha, \beta}$-(semi)stable. Note that by definition, any $E \in \Coh^{\beta}(X)$ satisfies $H^2 \cdot \ch_1^{\beta}(E) \geq 0$. Therefore, this function plays the same role in $\Coh^{\beta}(X)$ as the rank does in $\Coh(X)$.

As previously, Harder-Narasimhan filtrations exist. However, note that a version of Jordan-H\"older filtrations exists, but the stable factors are not unique up to order.

The notion of $2$-Gieseker stability occurs as a limit of tilt stability as follows.

\begin{prop}[{\cite[Proposition 14.2]{Bri08:stability_k3}}]
\label{prop:large_volume_limit}
Let $E \in \Db(X)$ and $\beta < \mu(E)$. Then $E \in \Coh^{\beta}(X)$ and $E$ is $\nu_{\alpha, \beta}$-(semi)stable for $\alpha \gg 0$ if and only if $E \in \Coh(X)$ and $E$ is $2$-Gieseker-(semi)stable.
\end{prop}

The statement in \cite{Bri08:stability_k3} is for K3 surfaces, but the same proof works in our setting. If $\beta > \mu(E)$ the situation is slightly more complicated. The following Proposition is a combination of \cite[Lemma 2.7]{BMS16:abelian_threefolds} and \cite[Proposition 3.1]{LM16:examples_tilt}.

\begin{prop}
\label{prop:negative_rank_large_volume_limit}
Take a $\nu_{\alpha, \beta}$-semistable object $E \in \Coh^{\beta}(X)$. If $\beta \neq \mu(E)$, then $\HH^{-1}(E)$ is a reflexive sheaf, and if $\beta \geq \mu(E)$ and $\alpha \gg 0$, then $\HH^{-1}(E)$ is a torsion-free slope-semistable sheaf and $\HH^0(E)$ is supported in dimension less than or equal to one. 
\end{prop}

Semistable sheaves satisfy the Bogomolov inequality (see \cite[Theorem 3.4.1]{HL10:moduli_sheaves}). A version for tilt stability was proved in \cite[Corollary 7.3.2]{BMT14:stability_threefolds}.

\begin{thm}[Bogomolov inequality]
Let $E \in \Coh^{\beta}(X)$ be $\nu_{\alpha, \beta}$-semistable. Then
\[
\Delta_H(E) \coloneqq (H^2 \cdot \ch_1(E))^2 - 2 (H^3 \cdot \ch_0(E))(H \cdot \ch_2(E)) \geq 0.
\]
\end{thm}



Most applications of tilt stability come from varying $(\alpha, \beta)$ and determining what that means for the stability of a given set of objects. We visualize the parameter space of tilt stability,  $(\alpha, \beta) \in \R^2$ with $\alpha > 0$, as the upper half-plane via $i\alpha + \beta$. For a given class $v \in K_0(X)$ it turns out that there is a locally finite wall and chamber structure such that stability only changes as we cross a wall. These walls are either semicircles with center on the $\beta$-axis or vertical lines, see Figure \ref{fig:walls} and Figure \ref{fig:walls_torsion}. In the following, we recall what this means formally.

For $v \in K_0(X)$ we write $\ch(v)$, $\mu(v)$, $\nu_{\alpha, \beta}(v)$, and $\Delta(v)$ to mean the appropriate versions where $E$ is replaced by $v$.

\begin{defn}
For $v, w \in K_0(X)$ we define
\[
W(v, w) \coloneqq \{ (\alpha, \beta) \in \R_{>0} \times \R \colon \nu_{\alpha, \beta}(v) = \nu_{\alpha, \beta}(w) \}.
\]
The set $W(v, w)$ is a \emph{numerical wall} if $W(v, w) \neq \emptyset$ and $W(v, w) \neq \R_{>0} \times \R$, i.e., if it is a proper non-trivial subset of the upper half-plane.
\end{defn}

Numerical walls in tilt stability have a rather simple structure as shown in \cite{Mac14:nested_wall_theorem}:

\begin{thm}[Nested wall theorem]
\label{thm:nested_wall_thm}
Let $v \in K_0(X)$ with $\Delta(v) \geq 0$.
\begin{enumerate}
    \item A numerical wall for $v$ is either a semicircle centered along the $\beta$-axis or a vertical line parallel to the $\alpha$-axis in the upper half plane
    \item If $\ch_0(v) \neq 0$, then there is a unique numerical vertical wall for $v$ given by $\beta = \mu(v)$. The remaining numerical walls  for $v$ are split into two sets of nested semicircles whose apexes lie on the hyperbola $\nu_{\alpha, \beta}(v) = 0$. In particular, no two distinct walls intersect.
    \item If $\ch_0(v) = 0$ and $H^2 \cdot \ch_1(v) \neq 0$, then every numerical wall for $v$ is a semicircle whose apex lies on the ray $\beta = (H \cdot \ch_2(v))/(H^2 \cdot \ch_1(v))$.
\end{enumerate}
\end{thm}

The following is a well-known consequence of the fact that walls do not intersect.

\begin{cor}
Let
\[
0 \to F \to E \to G \to 0
\]
be a short exact sequence of $\nu_{\alpha, \beta}$-semistable objects in $\Coh^{\beta_0}(X)$ for some $(\alpha_0, \beta_0) \in W(F, E)$. Then this is a short exact sequence of $\nu_{\alpha, \beta}$-semistable object in $\Coh^{\beta}(X)$ for any $(\alpha, \beta) \in W(E, F)$.
\end{cor}

\begin{defn}
Let $v \in K_0(X)$. A numerical wall $W$ for $v$ is called an \emph{actual wall} for $v$ if there is a short exact sequence
\[
0 \to F \to E \to G \to 0
\]
of $\nu_{\alpha, \beta}$-semistable objects in $\Coh^{\beta}(X)$ for one $(\alpha, \beta) \in W(F, E)$ such that $W = W(F, E)$ and $\ch(E) = v$.
\end{defn}

Note that the above Corollary implies that this is a short exact sequence in $\Coh^{\beta}(X)$ for all $(\alpha, \beta) \in W(F, E)$. Determining walls is the key technique in this paper. It will allow us to classify sheaves with certain Chern characters in terms of short exact sequences (see Theorem \ref{thm:moduli_rank_three_set_theoretic}). Note that the condition $W(F, E) \neq \mathbb{R}_{>0} \times \mathbb{R}$ implies $\nu_{\alpha, \beta}(F) > \nu_{\alpha, \beta}(E)$ on one side of such a wall.  We say that that the short exact sequence
\[
0 \to F \to E \to G \to 0,
\]
or sometimes the wall $W(F, E)$, \emph{destabilizes} $E$.

\begin{figure}[htp]
\centering
\includegraphics[width=16cm]{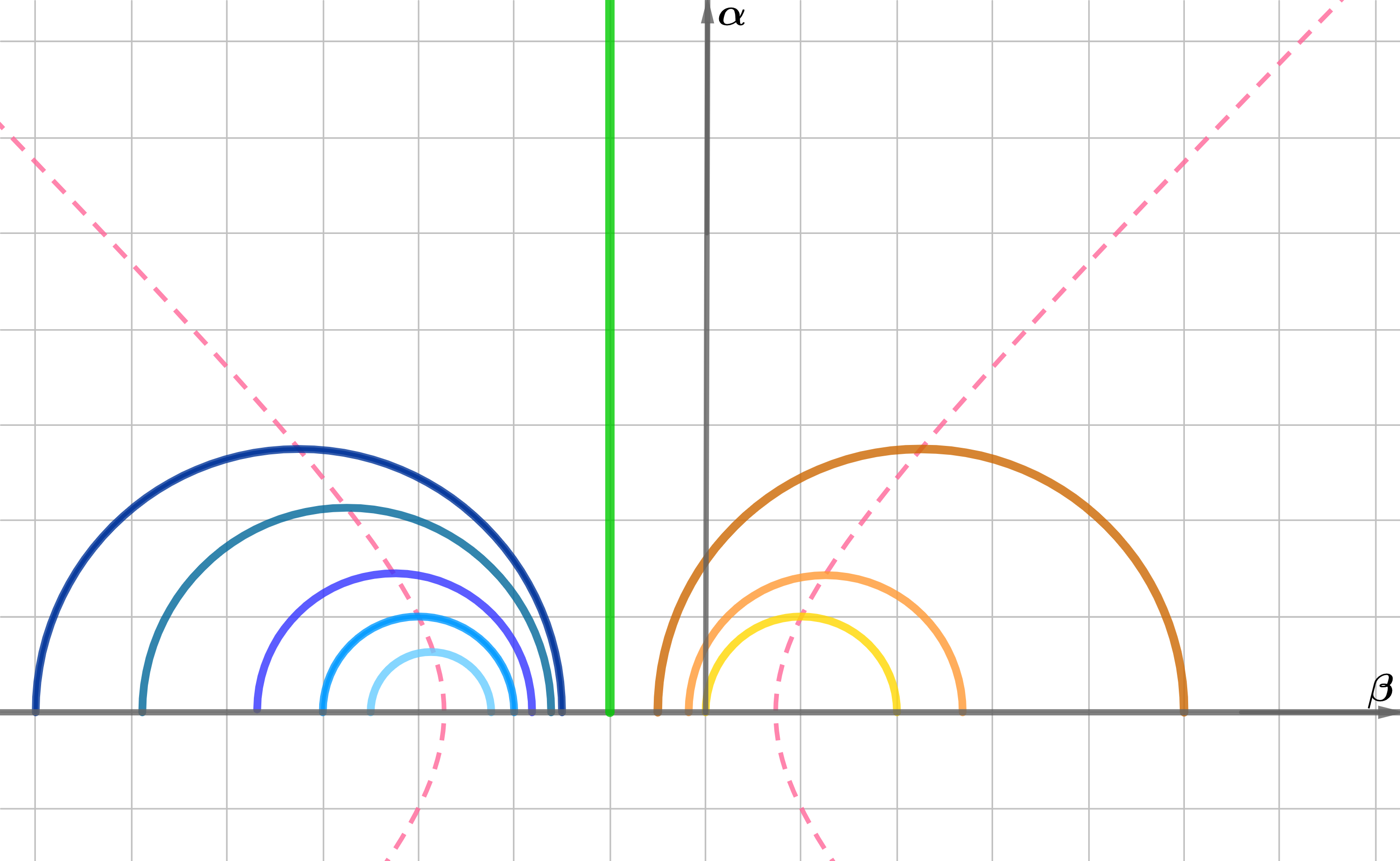}
\caption{Walls are nested semicircles or a unique vertical wall (Theorem \ref{thm:nested_wall_thm} (ii))}
\label{fig:walls}
\end{figure}

\begin{figure}[htp]
\centering
\includegraphics[width=16cm]{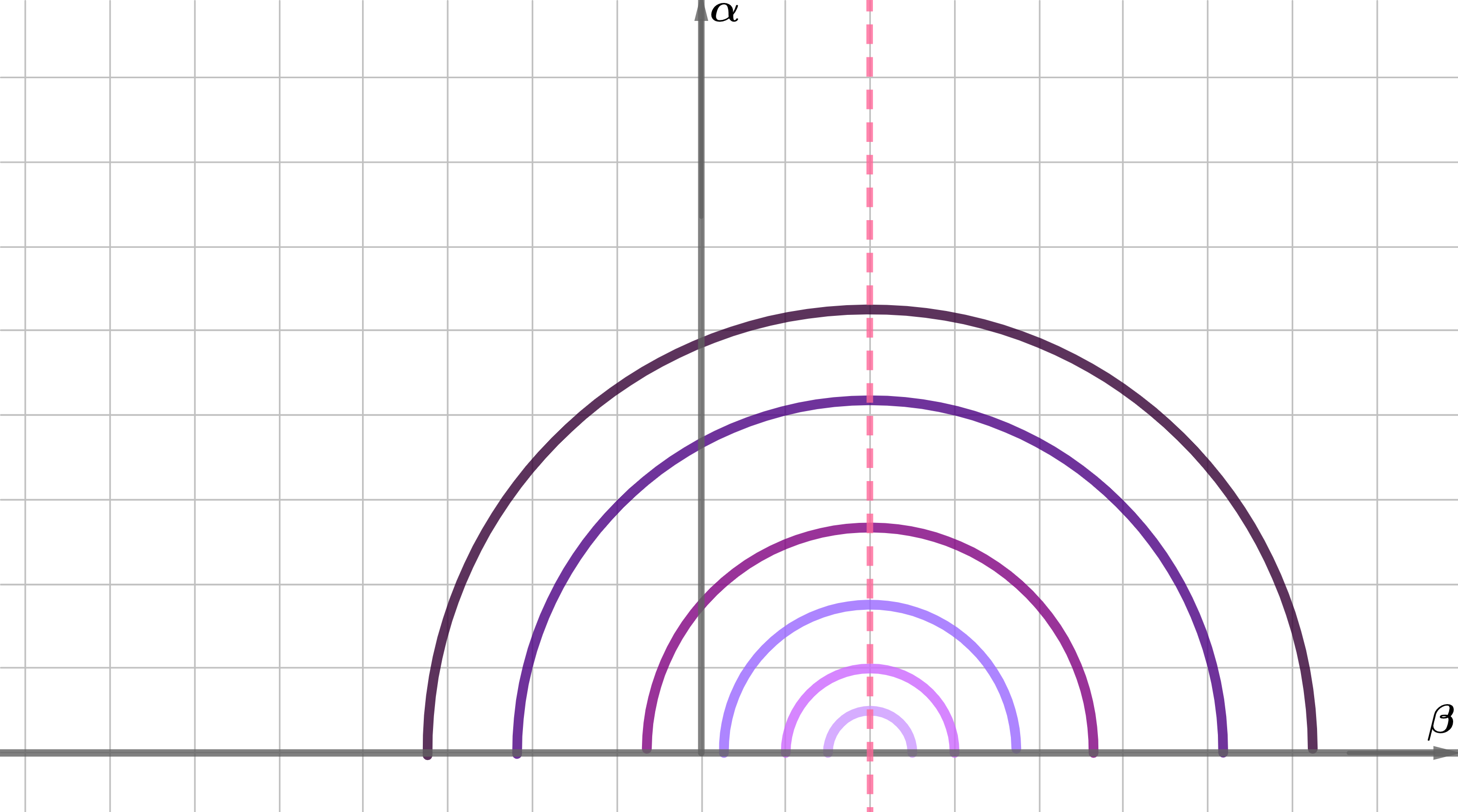}
\caption{Walls are nested semicircles (Theorem \ref{thm:nested_wall_thm} (iii))}
\label{fig:walls_torsion}
\end{figure}

\begin{prop}[{\cite[Appendix A]{BMS16:abelian_threefolds}}]
\label{prop:Deltabound}
If an actual wall is induced by a short exact sequence of tilt-semistable objects $0 \to F \to E \to G \to 0$, then
\[
\Delta_H(F) + \Delta_H(G) \leq \Delta_H(E),
\]
and equality can only occur if either $F$ or $G$ is a sheaf supported in dimension zero.
\end{prop}

It turns out that walls of large radius can only be induced by subobjects of small rank. The following precise statement is close to \cite[Proposition 8.3]{CH16:ample_cone_plane}. A proof of this version can be found in \cite[Lemma 2.4]{MS18:space_curves} for the case of non-negative ranks. The case of non-positive ranks has the exact same proof with reversed signs.

\begin{prop}
\label{prop:bounding_rank}
Assume that an object $E$ is destabilized by a semicircular wall induced by a subobject $F \into E$ or quotient $E \onto F$ with $\ch_0(F) > \ch_0(E) \geq 0$ or $\ch_0(F) < \ch_0(E) \leq 0$. Then the radius $\rho$ of $W(F, E)$ satisfies
\[
\rho^2 \leq \frac{\Delta_H(E)}{4 (H^3 \cdot \ch_0(F))(H^3 \cdot \ch_0(F) - H^3 \cdot \ch_0(E))}.
\]
\end{prop}



Tilt stability interacts nicely with the derived dual $\D(\cdot) \coloneqq \RHom( \cdot, \OO_X)[1]$.

\begin{prop}[{\cite[Proposition 5.1.3]{BMT14:stability_threefolds}}]
\label{prop:derived_dual}
Let $E \in \Coh^{\beta}(X)$ be a $\nu_{\alpha, \beta}$-semistable object with $\nu_{\alpha, \beta}(E) \neq \infty$. Then there is a $\nu_{\alpha, -\beta}$-semistable objects $\widetilde{E} \in \Coh^{-\beta}(X)$, a torsion sheaf $T$ supported in dimension zero, and a distinguished triangle
\[
\widetilde{E} \to \D(E) \to T[-1] \to \widetilde{E}[1].
\]
\end{prop}

The following proposition seems to be well known to experts, but we could find no proof in the literature.

\begin{prop}
\label{prop:vertical_wall_destabilized}
Let $E \in \Coh(X)$ be torsion-free. Then $E[1]$ is tilt-stable along the vertical wall $\beta = \mu(E)$ if and only if $E$ is slope-stable and reflexive. In particular, slope-stable reflexive sheaves do not get destabilized along the vertical wall.
\end{prop}

\begin{proof}
If $E$ is slope-unstable, then $E \notin \Coh^{\mu(E)}(X)$. Assume that $E$ is strictly slope-semistable. Then there is a short exact sequence of slope-semistable sheaves 
\[
0 \to F \to E \to G \to 0
\]
such that $\mu(F) = \mu(G)$. Taking a shift by one, this becomes a short exact sequence in $\Coh^{\mu(E)}(X)$ with $\nu_{\alpha, \mu(E)}(F[1]) = \nu_{\alpha, \mu(E)}(G[1])$.

Assume that $E$ is not reflexive, but slope-stable. Then we have a short exact sequence in $\Coh^{\mu(E)}(X)$ given by
\[
0 \to T \to E[1] \to E^{\vee \vee}[1] \to 0
\]
where $T$ is a non-trivial sheaf supported in dimension less than or equal to one. However, this sequence makes $E[1]$ strictly tilt-semistable along $\beta = \mu(E)$.

Assume vice-versa that $E$ is a slope-semistable reflexive sheaf. Then it is an object in $\Coh^{\mu(E)}(X)$ of maximal phase, and in particular tilt-semistable. If it is strictly semistable, then it admits a short exact sequence
\[
0 \to F \to E[1] \to G[1] \to 0
\]
where $F$, $G[1]$, $\HH^{-1}(F)[1]$, and $\HH^0(F)$ are also of maximal phase. In particular, $\HH^{-1}(F)$ and $G$ are torsion-free and slope-semistable of slope $\mu(E)$, and $\HH^0(F)$ has support of dimension at most one. 

Consider the long exact sequence 
\[
0 \to \HH^{-1}(F) \to E \to G \to \HH^0(F) \to 0.
\]
Since we assume that $E$ is strictly stable, this is a contradiction unless $\HH^{-1}(F) = 0$.
Taking duals we get an exact sequence
\[
0 \to G^{\vee} \to E^{\vee} \to \lExt^1(F, \OO_X). 
\]
Since $F$ is supported in dimension less than or equal to one, this implies $\lExt^1(F, \OO_X) = 0$ and $G^{\vee} \cong E^{\vee}$. Hence, $E \subsetneq G = G^{\vee \vee} = E^{\vee \vee}$, a contradiction to $E$ being reflexive.
\end{proof}

From now on, we assume $X \subset \P^4$ is a smooth cubic threefold. In the later sections, we need the following result of \cite[Proposition 3.2]{Li19:conjecture_fano_threefold} which improves Bogomolov inequality in the case of a Fano threefold of Picard rank one. Be aware that our notation differs from Li's. 

\begin{thm}
\label{thm:li_delta_bound}
Let $E$ be a tilt-stable with $\ch_0(E) \neq 0$ for some $\alpha > 0$, $\beta \in \R$. If $-\tfrac{1}{2} \leq \mu_H(E) \leq \tfrac{1}{2}$, then $\tfrac{H \cdot \ch_2(E)}{H^3 \cdot \ch_0(E)} \leq 0$.  
\end{thm}

In the case of cubic threefolds, direct sums of line bundles can be detected among semistable sheaves or objects by their Chern characters as follows.

\begin{prop}
\label{prop:line_bundles}
\begin{enumerate}
    \item \label{enum:beta<0} If $E$ is slope-semistable or $\nu_{\alpha, \beta}$-semistable for some $\alpha > 0$, $\beta < 0$ with $\ch(E) = (r, 0, 0, eH^3)$ where $r > 0$, then $e \leq 0$. If additionally, $e = 0$, then $E \cong \OO_X^{\oplus r}$.
    \item \label{enum:beta>0} If $E$ is $\nu_{\alpha, \beta}$-semistable for some $\alpha > 0$, $\beta > 0$ with $\ch(E) = (-r, 0, 0, eH^3)$ where $r > 0$, then $e = 0$ and $E \cong \OO_X^{\oplus r}[1]$.
\end{enumerate}
\end{prop}

\begin{proof}
In either case, Proposition \ref{prop:Deltabound} and $\Delta(E) = 0$ imply that $E$ has no semicircular walls.

We first claim that the only slope-stable reflexive sheaf of class $(r, 0, 0, eH^3)$ is $\OO_X$. Assume otherwise. By Proposition \ref{prop:vertical_wall_destabilized}, such an $E$ is also stable at the vertical wall $\beta = 0$, and thus, it is $\nu_{\alpha, \beta}$-stable for all $\alpha >0, \beta \in \R$.  Since $\nu_{0, \beta}(E) = -\frac{\beta}2 > -\frac{\beta}2 - 1 = \nu_{0, \beta}(\OO_X(-2H)[1])$ and both objects are stable for $\alpha \ll 1 $ and $\beta \in (-2, 0)$, we have
$\Ext^2(\OO_X, E) = \Hom(E, \OO_X(-2H)[1]) = 0$. Similarly, from $\nu_{\alpha, \beta}$-stability for $\alpha \ll 1$ and $\beta \in (0, 2)$ we obtain 
$\Ext^2(E, \OO_X) = \Hom(\OO_X(2H), E[1]) = 0$. However, at least one of $\chi(\OO_X, E) = r + 3e$ or $\chi(E, \OO_X) = r - 3e$ is positive, and so $E$ admits a morphism from $\OO_X$ or a morphism to $\OO_X$. As both are reflexive and slope-stable of slope 0, this shows $E \cong \OO_X$.

Now consider an object $E$ as in case \ref{enum:beta<0}. Then $E[1]$ is $\nu_{\alpha, 0}$-semistable.
By Proposition \ref{prop:vertical_wall_destabilized}, its Jordan-H\"older factors are either of the form $F[1]$ for a slope-stable reflexive sheaf $F$ with $\ch(F) = (r_F, 0, d_F H^2, e_F H^3)$, or a torsion sheaf supported in dimension $\le 1$. In fact, Proposition \ref{prop:Deltabound} shows $d_F = 0$ in the former case, and thus, $F = \OO_X$ by the previous case, and that the torsion sheaves are supported in dimension zero. As $-3e$ is the total length of the torsion sheaves, we get $e \leq 0$. If $e = 0$, all factors are isomorphic to $\OO_X[1]$ and the claim follows from $\Ext^1(\OO_X, \OO_X) = 0$.  

In case $\ref{enum:beta>0}$, we again consider a Jordan-H\"older filtration with respect to $\nu_{\alpha, 0}$-stability. Let $E_i \into E_{i+1}$ be the first filtration step where the quotient $E_{i+1}/E_i$ is a zero-dimensional torsion sheaf $T$, should one exist. Then $E_i = \OO_X[1]^{\oplus k}$ for some $k > 0$. Since $\Ext^1(T, \OO_X[1]) = H^1(T)^\vee =  0$, we have $E_{i+1} = E_i \oplus T$, and so $T$ is a subobject of $E$. This contradicts stability of $E$ for $\beta > 0$. Thus, $E = \OO_X[1]^{\oplus r}$ as claimed.
\end{proof}


\section{Construction of sheaves}
\label{sec:construction_sheaves}

In this section, we introduce the sheaves that make up our moduli space $\overline{M}_X(v)$. It turns out that all of them are at least reflexive, and the generic one is a vector bundle. From now on $X \subset \P^4$ is an arbitrary smooth cubic threefold.

Let $Y \subset X$ be an arbitrary hyperplane section, $D$ be an effective Weil divisor on $Y$, and $V \subset H^0(\OO_Y(D))$ be a non-trivial subspace. Then we define $\EE_{D, V} \in \Db(X)$ to be the cone of the induced morphism $\OO_X \otimes V \to \OO_Y(D)$. Moreover, let $E_{D, V} \coloneqq \HH^{-1}(\EE_{D, V})$. Hence, we have a long exact sequence
\[
0 \to E_{D, V} \to \OO_X \otimes V \to \OO_Y(D)\to \HH^{0}(\EE_{D, V}) \to 0.
\]
If $V = H^0(\OO_Y(D))$, we will drop $V$, and write $\EE_D$, respectively $E_D$.

\begin{lem}
\label{lem:E_D_vector_bundle}
The sheaf $E_{D, V}$ is slope-stable and reflexive. If additionally $\HH^0(\EE_{D, V}) = 0$, then $E_{D, V}$ is a vector bundle.
\end{lem}

\begin{proof}
The quotient $(\OO_X \otimes V) / E_{D, V}$ embeds into $\OO_Y(D)$. Since $Y$ is integral by Proposition \ref{prop:normal_hyperplane_section}, the sheaf $(\OO_X \otimes V) / E_{D, V}$ must be supported on $Y$. Therefore, $\ch_{\leq 1}(E_{D, V}) = (\dim V, -H)$ is primitive and it is enough to show that $E_{D, V}$ is slope-semistable. If not, let $F \subset E_{D, V}$ be the slope-semistable subsheaf in the Harder-Narasimhan filtration of $E_{D, V}$. Then $\mu(F) > \mu(E_{D, V})$ and the quotient $E_{D, V}/F$ is torsion-free. Since $F$ is also a subsheaf of $\OO_X \otimes V$, we must have $\mu(F) = 0$. Let $\ch(F) = (r, 0, dH^2, eH^3)$. The quotient $(\OO_X \otimes V)/F$ satisfies $\ch((\OO_X \otimes V)/F) = (\dim V - r, 0, -dH^2, -eH^3)$. By the Snake Lemma this quotient is either torsion-free or has a torsion subsheaf purely supported on $Y$. However, if it is not torsion-free, then its torsion-free quotient would destabilize $\OO_X \otimes V$, a contradiction. As a torsion-free quotient of $\OO_X \otimes V$ with slope zero, $(\OO_X \otimes V)/F$ has to be slope-semistable as well.

The classical Bogomolov inequalities $\Delta_H(F) \geq 0$ and $\Delta_H((\OO_X \otimes V)/F) \geq 0$ imply $d = 0$. Applying Proposition \ref{prop:line_bundles} to both $F$ and $(\OO_X \otimes V)/F$ implies $e = 0$, and finally, $F = \OO_X^{\oplus r}$. However, by construction, $E_{D, V}$ has no global sections, a contradiction.

To see that $E_{D, V}$ is reflexive it suffices to show that $\lExt^q(E_{D, V}, \OO_X) = 0$ for $q \geq 2$ and that $\lExt^1(E_{D, V}, \OO_X)$ is supported in dimension zero. If additionally $\lExt^1(E_{D, V}, \OO_X) = 0$, then $E_{D, V}$ is a vector bundle.

Clearly, $\lExt^q(\OO_X \otimes V, \OO_X) = 0$ for $q \neq 0$. Because $\OO_Y(D)$ is a rank one reflexive sheaf on the codimension one subvariety $Y$, the quotient $(\OO_X \otimes V) / E_{D, V} \subset \OO_Y(D)$ is purely supported on $Y$. We can use \cite[Proposition 1.1.10]{HL10:moduli_sheaves} to see that $\lExt^q((\OO_X \otimes V) / E_{D, V}, \OO_X) = 0$ for all $q \neq 1, 2$, and $\lExt^2((\OO_X \otimes V) / E_{D, V}, \OO_X)$ is supported in dimension zero. The long exact sequence obtained from dualizing the short exact sequence
\begin{equation}\label{E-D}
    0 \to E_{D, V} \to \OO_X \otimes V \to (\OO_X \otimes V) / E_{D, V} \to 0
\end{equation}
implies the required vanishings.

If additionally $\HH^0(\EE_{D, V}) = 0$, then $(\OO_X \otimes V) / E_{D, V} = \OO_Y(D)$ is a reflexive sheaf on the codimension one subvariety $Y$, and we can use \cite[Proposition 1.1.10]{HL10:moduli_sheaves} again to see that $\lExt^2(\OO_Y(D), \OO_X) = 0$. The same long exact sequence as above now implies $\lExt^1(E_{D, V}, \OO_X) = 0$.
\end{proof}

Note that we will use this Lemma for the case $\ch(\OO_Y(D)) = (0, H, H^2/2, -H^3/6)$. It will turn out that in this case $h^0(\OO_Y(D)) = 3$ for any such $D$, see Theorem \ref{thm:moduli_rank_three_set_theoretic}, and we will choose $V = H^0(\OO_Y(D))$. Moreover, we will show that in that case $\HH^{0}(\EE_D) = 0$, i.e., $\OO_Y(D)$ is globally generated, see Theorem \ref{thm:moduli_rank_three_set_theoretic}. A straightforward computation shows that in this example $\ch(E_D) = (3,-H,-H^2/2,H^3/6)$.

\begin{cor}
\label{cor:k_p_stable}
Let $P \in X$. Then $h^0(\II_P(H)) = 4$ and the sheaf $K_P$ defined through the exact sequence
\begin{equation}\label{k-p}
0 \to K_P \to \OO_X^{\oplus 4} \to \II_P(H) \to 0
\end{equation}
satisfies $\ch(K_P) = (3,-H,-H^2/2,H^3/6)$. Moreover, $K_P$ is reflexive and slope-stable, and locally free except at $P$.
\end{cor}

\begin{proof}
By choosing an embedding $K_P \into \OO_X^{\oplus 3}$ we get a short exact sequence
\[
0 \to K_P \to \OO_X^{\oplus 3} \to \II_{P/Y}(H) \to 0
\]
for some hyperplane section $Y$. The statement then follows from Lemma \ref{lem:E_D_vector_bundle} by choosing $D = H$ and $V = H^0(\II_{P/Y}(H)) \subset H^0(\OO_Y(H))$.

From the defining short exact sequence \eqref{k-p} one immediately sees that $K_P$ is locally free away from $P$ (as it is the kernel of a surjective map of vector bundles), and not locally free at $P$ (as $\Ext^2(\OO_P, K_P) = \Ext^1(\OO_P, I_P(H)), \neq 0$).
\end{proof}


\section{Variation of stability}
\label{sec:wallstructure}

In this section, we investigate semistable sheaves with Chern character 
\[
v \coloneqq \left(3, -H, -\frac{1}{2} H^2, \frac{1}{6} H^3 \right).
\]
The main goal is to use wall-crossing to prove the following Theorem, which gives a set-theoretic description of the moduli space $\overline{M}_X(v)$.

\begin{thm}
\label{thm:moduli_rank_three_set_theoretic}
\begin{enumerate}
    \item Let $D$ be Weil divisor on a (possibly singular) hyperplane section $Y$ with $\ch(\OO_Y(D)) = (0, H, \tfrac{1}{2} H^2, -\tfrac{1}{6} H^3)$. Then $\OO_Y(D)$ is globally generated, and $h^0(\OO_Y(D)) = 3$. In particular, there exists a smooth twisted cubic $C$ in $Y$ of class $D$.
    \item A sheaf $E$ with Chern character $v$ is Gieseker-semistable if and only if it is either equal to the reflexive sheaf $K_P$ for a point $P \in X$ \eqref{k-p}, or the vector bundle $E_D$ for a Weil divisor $D$ on a hyperplane section $Y \subset X$  \eqref{E-D} with $\ch(\OO_Y(D)) = (0, H, \tfrac{1}{2} H^2, -\tfrac{1}{6} H^3)$.
\end{enumerate}
\end{thm}

Note that since $\ch_1(E) = -H$, any Gieseker-semistable sheaf of class $v$ is slope stable. The argument will essentially boil down to a detailed analysis of the numerical wall $W$ defined by
\begin{equation}
\label{eq:main_wall}
\alpha^2 + \left(\beta - \frac{1}{2}\right)^2 = \frac{1}{4}. 
\end{equation}
At this wall, the short exact sequences \eqref{k-p} and \eqref{E-D} become destabilizing short exact sequences in $\Coh^\beta(X)$ in the form
$0 \to \OO_Y(D) \to E_D[1] \to \OO_X[1]^{\oplus 3} \to 0$ and $0 \to I_P(H) \to K_P[1] \to \OO_X[1]^{\oplus 4} \to 0$. Moreover, we can show that every object gets destabilized, and the destabilizing short exact sequence must be of one of these types, see Lemma \ref{lem:wall_classification_rank_three}.


\subsection{Classification of some torsion sheaves}

In this section, we prove the following Proposition.

\begin{prop}
\label{prop:torsion_classification}
The wall $W$ of equation \eqref{eq:main_wall} is the unique actual wall in tilt stability for objects $G$ with Chern character $\ch(G) = (0, H, \tfrac{1}{2} H^2, -\tfrac{1}{6} H^3)$.
\begin{enumerate}
    \item Above $W$ the moduli space of tilt-semistable objects is the moduli space of Gieseker-semistable sheaves, and contains precisely the following two types of sheaves $G$:
    \begin{enumerate}
        \item $G = \II_{P/Y}(H)$ for $Y \in |H|$ and $P \in Y$, and
        \item $G = \OO_Y(D)$ where $D$ is a Weil-divisor on some $Y \in |H|$.
    \end{enumerate}
    \item Below $W$ the moduli space of tilt-semistable objects contains precisely the following two types of objects $G$:
    \begin{enumerate}
        \item the unique non-trivial extensions
        \begin{equation}\label{G-p}
            0 \to \OO_X[1] \to G_P \to \II_P(H) \to 0
        \end{equation}
        for points $P \in X$, and
        \item $G = \OO_Y(D)$ where $D$ is a Weil-divisor on some $Y \in |H|$.
    \end{enumerate}
\end{enumerate}
\end{prop}

We start by dealing with slightly more general objects without fixing $\ch_3$.

\begin{lem}
\label{lem:torsion_wall_classification}
The wall $W$ of equation \eqref{eq:main_wall} is the unique actual wall in tilt stability for objects $G$ with Chern character $\ch_{\leq 2}(G) = (0, H, \tfrac{1}{2} H^2)$. If $G$ is strictly semistable along $W$, then any Jordan-H\"older filtration of $G$ is given by either
\[
0 \to \II_Z(H) \to G \to \OO_X[1] \to 0,
\]
or
\[
0 \to \OO_X[1] \to G \to \II_Z(H) \to 0,
\]
where $Z \subset X$ is a zero-dimensional subscheme of length $H^3/6 - \ch_3(G)$.
\end{lem}

\begin{proof}
All walls for $(0, H, \tfrac{1}{2} H^2)$ intersect the vertical ray $\beta = \tfrac{1}{2}$. If $G$ is strictly-semistable along some numerical wall intersecting $\beta = \tfrac{1}{2}$, then there is a short exact sequence in $\Coh^{1/2}(X)$ of tilt-semistable objects
\[
0 \to A \to G \to B \to 0
\]
with equal tilt-slope. Let $\ch_{\leq 2}(A) = (r, cH, dH^2)$. By definition of $\Coh^{1/2}(X)$ and the fact that neither $A$ nor $B$ can have infinite tilt-slope we get
\[
0 < H^2 \cdot \ch^{1/2}_1(A) = H^3 (c - \frac{r}{2}) < H^2 \cdot \ch^{1/2}_1(G) = H^3.
\]
Therefore, $c = \tfrac{r}{2} + \tfrac{1}{2}$, and in particular, $r$ is odd. We will deal with the case $r < 0$. If $r > 0$, then $B$ has negative rank and one simply has to exchange the roles of $A$ and $B$ in the following argument.

For $(\alpha, \tfrac{1}{2}) \in W(A, G)$ we have 
\[
-\alpha^2r + 2d - \frac{r}{4} - \frac{1}{2} = \nu_{\alpha, 1/2}(A) = \nu_{\alpha, 1/2}(G) = 0.
\]
Since $\alpha^2 > 0$, this implies $d < \tfrac{r}{8} + \tfrac{1}{4}$. The fact 
\[
0 \leq \frac{\Delta_H(A)}{(H^3)^2} = -2dr + \frac{r^2}{4} + \frac{r}{2} + \frac{1}{4}
\]
implies $d \geq \tfrac{r}{8} + \tfrac{1}{8r} + \tfrac{1}{4}$. Since $d \in \tfrac{1}{6} \Z$, these restrictions on $d$ are only possible for $r \in \{-1, -3\}$. 

If $r = -3$, then $\ch_{\leq 2}(A) = (-3, -H, -\tfrac{1}{6} H^2)$. This case is immediately ruled out by Theorem \ref{thm:li_delta_bound}. If $r = -1$, then $\ch_{\leq 2}(A) = (-1, 0, 0)$, and by Proposition \ref{prop:line_bundles}, we know $A = \OO_X[1]$. Then $\ch(B) = (1, H, \tfrac{1}{2} H^2, \ch_3(G))$. By Proposition \ref{prop:Deltabound}, we know that there is no semicircular wall for $B$, and by Proposition \ref{prop:large_volume_limit}, the object $B$ has to be a $2$-Gieseker-stable sheaf. Since $\ch(B(-H)) = (1, 0, 0, \ch_3(G) - \frac{1}{6} H^3)$, the remaining statement follows by applying Proposition \ref{prop:line_bundles} to $B(-H)$.
\end{proof}

The next step is to gain further control over the third Chern character.

\begin{lem}
\label{lem:torsion_bound_ch3}
Let $G$ be a $\nu_{\alpha, \beta}$-semistable object with $\ch_{\leq 2}(G) = (0, H, \tfrac{1}{2} H^2)$. Then $\ch_3(G) \leq \tfrac{1}{6} H^3$. If $\ch_3(G) = \tfrac{1}{6} H^3$ and $(\alpha, \beta)$ is above $W$, then $G \cong \OO_Y(H)$ for some $Y \in |H|$.
\end{lem}

\begin{proof}
We may assume $\ch_3(G) \geq \tfrac{1}{6} H^3$. By Lemma \ref{lem:torsion_wall_classification}, the only possible wall is given by $W$. Therefore, $G$ has to be tilt-semistable along $W$. Since $W$ lies below the numerical wall $W(G, \OO_X(-H)[1])$, we get $\ext^2(\OO_X(H), G) = \hom(G, \OO_X(-H)[1]) = 0$. Thus, $\hom(\OO_X(H), G) \geq \chi(\OO_X(H), G) = \ch_3(G) + \frac{1}{6} H^3 > 0$. Therefore, $W$ is a wall for $G$ and by Lemma \ref{lem:torsion_wall_classification}, the destabilising sequence is  
\[
0 \to \OO_X(H) \to G \to \OO_X[1] \to 0.
\]
This implies $G = \OO_Y(H)$ for some $Y \in |H|$ and $\ch_3(G) = \frac{1}{6} H^3$.
\end{proof}

\begin{proof}[Proof of Proposition \ref{prop:torsion_classification}]
Assume that $G$ is strictly tilt-semistable along $W$. Then Lemma \ref{lem:torsion_wall_classification} splits our problem into two cases.

Firstly, assume that $G$ fits into a non-splitting short exact sequence
\[
0 \to \II_P(H) \to G \to \OO_X[1] \to 0
\]
for a point $P \in X$. Then clearly $G = \II_{P/Y}(H)$ for some $Y \in |H|$. This object is tilt-stable above $W$, and tilt-unstable below $W$ by precisely this sequence.

Secondly, assume that $G$ fits into a non-splitting short exact sequence
\begin{equation}
\label{eq:object_below_wall}
0 \to \OO_X[1] \to G \to \II_P(H) \to 0
\end{equation}
for some $P \in X$. By Serre duality, $\Ext^1(\II_P(H), \OO_X[1]) = h^1(\II_P(-H)) = 1$ and hence, there is a unique $G$ for each $P \in X$. Clearly, this object is tilt-unstable above $W$. Assume it is also tilt-unstable below $W$. Then there is a short exact sequence $0 \to A \to G \to B \to 0$ destabilizing $G$ below the wall. However, $G$ is strictly-semistable at $W$, and by Lemma \ref{lem:torsion_wall_classification}, this implies $B = \OO_X[1]$. However, that means the short exact sequence \eqref{eq:object_below_wall} splits, a contradiction.

Lastly, assume that $G$ is $\nu_{\alpha, \beta}$-stable for all $(\alpha, \beta)$. By Proposition \ref{prop:derived_dual}, $\D(G)$ lies in a distinguished triangle 
\begin{equation}
\label{eq:triangle}
    \tilde{G} \to \D(G) \to T[-1] \to \tilde{G}[1]
\end{equation}
where $T$ is a torsion sheaf supported in dimension zero and $\tilde{G} \in \text{Coh}^{-\beta}(X)$ is $\nu_{\alpha, -\beta}$-semistable. If $\ch_3(T) = t$, then $\ch(\tilde{G}) = (0, H, -\tfrac{1}{2} H^2 , -\tfrac{1}{6} H^3 + t)$. Thus, $\tilde{G}$ is a pure sheaf supported on a hyperplane section $Y \in |H|$. We can compute 
\[
\ch(\tilde{G} \otimes \OO_X(H)) = \left(0, H, \frac{1}{2} H^2 , -\frac{1}{6} H^3 + t\right). 
\]
Thus, Lemma \ref{lem:torsion_bound_ch3} gives $t = 0$ or $t = 1$, and if $t = 1$, then $\tilde{G} \otimes \OO_X(H) \cong \OO_Y(H)$, i.e., $\tilde{G} \cong \OO_Y(H)$. Hence there is a non-trivial morphism $\OO_X \rightarrow \tilde{G}$. Since $\hom(\OO_X, T[-i]) = 0$ for $i>0$, The triangle \eqref{eq:triangle} shows that there is a non-trivial morphism $\OO_X \rightarrow \D(G)$. Dualizing this morphism leads to a non-trivial morphism $G \to \OO_X[1]$. However, this is in contradiction to the assumption that $G$ is stable along $W$. 

If $t = 0$, then $\D(G) = \tilde{G}$ is a sheaf, so $\lExt^q(G, \OO_X) = 0$ for $q > 1$. Thus, \cite[Proposition 1.1.10]{HL10:moduli_sheaves} implies
that $G$ is reflexive and supported on a hyperplane section $Y \in |H|$. This means $G = \OO_Y(D)$ for some Weil divisor $D$ on $Y$.
\end{proof}


\subsection{Set-theoretic description of the moduli space}

We now prepare the proof of Theorem \ref{thm:moduli_rank_three_set_theoretic}.

\begin{lem}
\label{lem:no_wall_at_minus_one}
There are no walls along $\beta = -1$ for tilt-semistable objects $E$ with Chern character $\ch_{\leq 2}(E) = (3, -H, -\tfrac{1}{2} H^2)$.
\end{lem}

\begin{proof}
Assume there is such a wall induced by a short exact sequence
\[
0 \to A \to E \to B \to 0
\]
with $\ch_{\leq 2}^{-1}(A) = (r, xH, yH^2)$. Then $0 < H \cdot \ch_1^{-1}(A) = xH^3 < H \cdot \ch_1^{-1}(E) = 2H^3$ implies $x = 1$. By exchanging the roles of $A$ and $B$ if necessary, we may assume $r \geq 2$.

Using $\Delta_H(A) \geq 0$ we get $y \leq \tfrac{1}{2r}$. A straightforward computation shows that there exists $\alpha > 0$ with 
$\nu_{\alpha, -1}(A) = \nu_{\alpha, -1}(E)$ if and only if $y > 0$. Since $y \in \tfrac{1}{6} \Z$, this is only possible if $y = \tfrac{1}{6}$ and $r \in \{2, 3\}$. Both cases $\ch_{\leq 2}^{-1}(A) = (3, H, \tfrac{1}{6} H^2)$ and $\ch_{\leq 2}^{-1}(A) = (2, H, \tfrac{1}{6} H^2)$ are directly ruled out by Theorem \ref{thm:li_delta_bound}.
\end{proof}

\begin{prop}
\label{prop:max_ch2_ch3_reflexive}
Take a slope-stable sheaf $E$ of Chern character $(3, -H, \ch_2, \ch_3)$. Then $H \cdot \ch_2 \leq -\tfrac{1}{2} H^3$, and if $\ch_2 \cdot H = -\frac{1}{2} H^3$, then $\ch_3 \leq \tfrac{1}{6} H^3$. In particular, this implies that any slope-stable sheaf of Chern character $v$ is a reflexive sheaf.
\end{prop}

\begin{proof}
Since $E$ is slope stable, the classical Bogomolov inequality gives 
\[
    \Delta_H(E) = (H^3)^2 -2(3H^3)\left( H \cdot \ch_2(E)\right) \geq 0 
\]
which implies $H \cdot \ch_2(E) \leq \tfrac{1}{6} H^3$. The case $H \cdot \ch_2(E) = \tfrac{1}{6} H^3$ is immediately
ruled out by Theorem \ref{thm:li_delta_bound}. 
Since $c_2(E) = \tfrac{1}{2} H^2 -\ch_2(E)$ has to be an integral class, we are left to rule out $H \cdot \ch_2(E) = -\frac{1}{6} H^3$. Assume $H \cdot \ch_2(E) = -\tfrac{1}{6} H^3$. We may assume that $E$ is a reflexive sheaf. If not, we replace it by the double dual $E^{\vee\vee}$ which satisfies $H \cdot \ch_2(E) \leq H \cdot \ch_2(E^{\vee\vee})$. By the first part of the argument $H \cdot \ch_2(E^{\vee\vee}) = -\tfrac{1}{6} H^3$ holds as well.

We first show that $\ext^2(E, E) = 0$. Since $H^3 \cdot \ch^{-1/2}_1(E) = \tfrac{1}{2} H$, any destabilizing subobject $F \subset E$ along $\beta = -\tfrac{1}{2}$ must satisfy $H^3 \cdot \ch^{-1/2}_1(F) = \tfrac{1}{2} H$ or $H^3 \cdot \ch^{-1/2}_1(F) = 0$. Thus, either $F$ or the quotient $E/F$ have infinite tilt-slope, a contradiction. This means $E$ is $\nu_{\alpha, -1/2}$-stable for all $\alpha > 0$.

By Proposition \ref{prop:vertical_wall_destabilized}, the object $E[1]$ is tilt-stable for $\beta = 0$ and $\alpha \gg 0$. Since $H^3 \cdot \ch_1(E[1]) = H^3$, the same type of argument as above shows that there cannot be any wall along $\beta = 0$. Hence, $E(-2H)[1]$ is $\nu_{\alpha, \beta}$-stable for $\beta = -2$ and any $\alpha > 0$.

A straightforward computation shows that $W(E, E(-2H)[1])$ intersects both the vertical lines $\beta = -2$ and $\beta = -\tfrac{1}{2}$. Therefore, $E$ and $E(-2H)[1]$ are tilt-stable for any $(\alpha, \beta) \in W(E, E(-2H)[1])$ and have the same phase, and thus, $\ext^2(E, E) = \hom(E, E(-2H)[1]) = 0$. Since $E$ is stable, we know $\hom(E, E) = 1$ and hence, $3 = \chi(E, E) = 1 - \ext^1(E, E) - \ext^3(E, E) \leq 1$, a contradiction.

Now assume $H \cdot \ch_2 = -\frac{1}{2}H^3$. We know $E \in \Coh^{\beta}(X)$ is $\nu_{\alpha, \beta}$-stable for $\alpha \gg 0$ and $\beta < -\tfrac{1}{3}$. By Lemma \ref{lem:no_wall_at_minus_one}, we have that $E$ is $\nu_{\alpha, -1}$-stable for any $\alpha > 0$. One can easily compute
\[
\nu_{0, -1}(\OO_X(-2H)[1]) < \nu_{0, -1}(E)
\]
which implies $h^2(E) = \hom(E, \OO_X(-2H)[1]) = 0$. Moreover, since $\mu(E) = -\tfrac{1}{3} < \mu(\OO_X)$, we get $\hom(\OO_X, E) = 0$. Therefore, $\chi(E) = \ch_3(E) - \frac{1}{6} H^3 \leq 0$ as claimed.

Lastly, assume that a slope-stable sheaf $E$ of Chern character $v$ is not reflexive. We have a short exact sequence
\[
0 \to E \to E^{\vee \vee} \to T \to 0.
\]
Since $E^{\vee \vee}$ is also slope-stable, and both $H \cdot \ch_2(E)$ and $H \cdot \ch_3(E)$ are maximal, one gets $\ch(E) = \ch(E^{\vee \vee})$. This is only possible if $T = 0$.
\end{proof}

To prove Theorem \ref{thm:moduli_rank_three_set_theoretic}, we start in the large volume limit.

\begin{lem}
\label{lem:large_volume_limit_v}
Take $\beta > -\frac{1}{3}$. An object $\tilde{E} \in \Coh^{\beta}(X)$ of Chern character $-v$ is $\nu_{\alpha, \beta}$-semistable for $\alpha \gg 0$ if and only if $\tilde{E} \cong E[1]$ for a slope-stable reflexive sheaf $E$.
\end{lem}

\begin{proof}
Take a $\nu_{\alpha, \beta}$-semistable object $\tilde{E}$ of class $-v$. Proposition \ref{prop:negative_rank_large_volume_limit} implies that $\HH^{-1}(\tilde{E})$ is a slope-stable reflexive sheaf and $\HH^0(\tilde{E})$ is a torsion sheaf supported in dimension $\leq 1$. Therefore,
\[
\ch(\HH^{-1}(\tilde{E})) = \left(3, -H, - \frac{1}{2} H^2 + \ch_2(\HH^0(\tilde{E})), \frac{1}{6} H^3 + \ch_3(\HH^0(\tilde{E}))\right).
\]
By Proposition \ref{prop:max_ch2_ch3_reflexive}, this is only possible if $\ch_2(\HH^0(\tilde{E})) = \ch_3(\HH^0(\tilde{E})) = 0$, i.e., $\HH^0(\tilde{E})$ = 0.

Conversely, any slope-stable reflexive sheaf $E$ of class $v$ is $\nu_{\alpha , \beta}$-stable for $\alpha \gg 0$ and $\beta < \mu(E) = -\frac{1}{3}$. Proposition \ref{prop:vertical_wall_destabilized} implies that $E[1]$ is $\nu_{\alpha , \beta}$-stable for $\alpha \gg 0$ and $\beta > \mu(E) = -\frac{1}{3}$.
\end{proof}

Next, we move down from the large volume limit and investigate walls for objects of class $-v$. Note that all walls to the right of the vertical wall must intersect $\beta = -\tfrac{1}{3}$.

\begin{lem}
\label{lem:wall_classification_rank_three}
The wall $W$ of equation \eqref{eq:main_wall} is the unique actual wall for objects with Chern character $-v$ to the right of the vertical wall. There are no tilt-semistable objects below $W$. Any tilt-semistable $\tilde{E}$ with Chern character $-v$ fits into one of the following two cases:
\begin{enumerate}
    \item $\tilde{E}$ fits into a short exact sequence
    \[
    0 \to \OO_Y(D) \to \tilde{E} \to \OO_X^{\oplus 3}[1] \to 0
    \]
    where $D$ is a Weil divisor on hyperplane section $Y \in |H|$;
    \item $\tilde{E}$ fits into a short exact sequence
    \[
    0 \to \II_P(H) \to \tilde{E} \to \OO_X^{\oplus 4}[1] \to 0
    \]
    where $P \in X$.
\end{enumerate}
\end{lem}

\begin{proof}
Let $\tilde{E}$ be a tilt-semistable object with Chern character $-v$. Let $W'$ be a wall strictly above $W$ induced by a short exact sequence $0 \to F \to \tilde{E} \to G \to 0$. Then the wall $W'$ contains points $(\alpha, 0)$ with $\alpha > 0$. In particular, $0 < H \cdot \ch_1(F) < H \cdot \ch_1(\tilde{E}) = H^3$, a contradiction.

Since the wall $W(\OO_X(2H), \tilde{E})$ is larger than $W$, we get $\hom(\tilde{E}, \OO_X[3]) = \hom(\OO_X(2H), \tilde{E}) = 0$ and
\[
\hom(\tilde{E}, \OO_X[1]) = \hom(\tilde{E} , \OO_X) + \ext^2(\tilde{E}, \OO_X) - \chi(\tilde{E} , \OO_X) \geq -\chi(\tilde{E} , \OO_X) = 3.
\]
Clearly, any morphism $\tilde{E} \to \OO_X[1]$ destabilizes $\tilde{E}$ below $W$.

Let $r \coloneqq \hom(\tilde{E}, \OO_X[1]) \geq 3$. We get a short exact sequence of tilt-semistable objects along $W$ given by
\[
0 \to G \to \tilde{E} \to \OO_X^{\oplus r}[1] \to 0.
\]
If $r \geq 4$, then Proposition \ref{prop:bounding_rank} says
\[
\frac{1}{4} \leq \frac{1}{r(r-3)}, 
\]
i.e., $r \leq 4$. For $r = 4$, we get $\ch(G(-H)) = (1, 0, 0, -\tfrac{1}{3} H^3)$ and so $G = \II_P(H)$ for some $P \in X$.

If $r = 3$, then $\ch(G) = (0, H, \tfrac{1}{2} H^2, -\tfrac{1}{6} H^3)$. Assume $G$ is not of the form $\OO_Y(D)$ for some Weil-divisor $D$ on a hyperplane section $Y \in |H|$. Then Proposition \ref{prop:torsion_classification} implies that $G$ has to be strictly-semistable along our wall $W$. Since $\tilde{E}$ is tilt-semistable above the wall, we know $\Hom(\OO_X[1], E) = 0$. Therefore, Lemma \ref{lem:torsion_wall_classification} shows that there is a short exact sequence
\[
0 \to \II_P(H) \to G \to \OO_X[1] \to 0
\]
for a point $P \in X$. But then there is an inclusion $\II_P(H) \into \tilde{E}$ and we are in the second case.
\end{proof}

\begin{proof}[Proof of Theorem \ref{thm:moduli_rank_three_set_theoretic}]
Let $D$ be a Weil-divisor on a hyperplane section $Y \in |H|$ with $\ch(\OO_Y(D)) = (0, H, \tfrac{1}{2}H^2, -\tfrac{1}{6} H^3)$. By Proposition \ref{prop:torsion_classification}, the sheaf $\OO_Y(D)$ is tilt-stable for all $\alpha > 0$, $\beta \in \R$. A straightforward computation shows that the numerical wall $W(\OO_Y(D), \OO_X(-2H)[1])$ is non-empty, and therefore, $h^2(\OO_Y(D)) = \hom(\OO_Y(D), \OO_X(-2H)[1]) = 0$. We conclude 
\[
h^0(\OO_Y(D)) = \chi(\OO_Y(D)) + h^1(\OO_Y(D)) + h^3(\OO_Y(D)) \geq \chi(\OO_Y(D)) = 3.
\]
We pick a three-dimensional subspace $V \subset h^0(\OO_Y(D))$ to get an object $\EE_{D, V} \in \Db(X)$ as in Section \ref{sec:construction_sheaves}. By Lemma \ref{lem:E_D_vector_bundle}, the sheaf $E_{D, V} = \HH^{-1}(\EE_{D, V})$ is slope-stable and reflexive. If $\HH^0(\EE_{D, V}) \neq 0$, then $E_{D, V}$ has a Chern character in contradiction to Proposition \ref{prop:max_ch2_ch3_reflexive}. This shows that $\OO_Y(D)$ is globally generated.

Since $E_{D, V}$ is slope-stable, we know $h^0(E_{D, V}) = 0$ and $h^3(E_{D, V}) = \hom(E_{D, V}, \OO_X(-2H)) = 0$. Moreover, as in the proof of Proposition \ref{prop:max_ch2_ch3_reflexive} we get $h^2(E_{D, V}) = 0$. This implies $h^1(E_{D, V}) = - \chi(E_{D, V}) = 0$. The long exact sequence obtained from taking sheaf cohomology of
\[
0 \to E_{D, V} \to \OO_X \otimes V \to \OO_Y(D) \to 0
\]
implies $H^i(\OO_Y(D)) = 0$ for $i > 0$ and $h^0(\OO_Y(D)) = 3$. Therefore, $V = H^0(\OO_Y(D))$ and for each $D$ there is a unique slope-stable sheaf $E_D = E_{D, V}$. 

Let $U \subset Y$ be the smooth locus of $Y$. By Lemma \ref{prop:normal_hyperplane_section}, we know that $Y$ is normal, and therefore, $Y \backslash U$ has dimension zero. In particular, a general section of $\OO_Y(D)$ leads to a curve completely contained in $U$. Since we work in characteristic $0$, we can use a version of Bertini's theorem \cite[Corollary III.10.9, Remark III.10.9.1, Remark III.10.9.2]{Har77:algebraic_geometry} on the open subset $U$ to see that a general section cuts out a smooth curve $C$. By adjunction, 
\begin{align*}
\ch(K_C)  &= \ch(\OO_Y(-H + D)_{|D}) = \ch(\OO_Y(-H + D)) - \ch(\OO_Y(-H)) \\
&= \Bigl(0, H, -\frac 12 H^2, - \frac 16 H^3\Bigr) - \Bigl(0, H, -\frac 32 H^2, \frac 76 H^3\Bigr) = \Bigl(0, 0, H^2, -\frac 43 H^3\Bigr),
\end{align*}
which shows that $C$ is of degree 3 with $\chi(K_C) = -1$, i.e., a twisted cubic. This completes the proof of part (i).

For part (ii), we already showed in Corollary \ref{cor:k_p_stable} that $K_P$ is slope-stable for any $P \in X$. Vice-versa, if $E$ is slope-stable, we can immediately conclude by Lemma \ref{lem:wall_classification_rank_three}.
\end{proof}

As a consequence we can already infer that our moduli space $\overline{M}_X(v)$ is smooth.

\begin{cor}
\label{cor:smooth_4dim}
Every Gieseker-semistable sheaf $E$ with $\ch(E) = (3, -H, -H^2/2, H^3/6)$ satisfies
\[
\Ext^i(E,E) =
\begin{cases}
\C &\text{if } i = 0 \\
\C^4 &\text{if } i =1 \\
0 &\text{otherwise.}
\end{cases}
\]
In particular, the moduli space $\overline{M_{X}(v)}$ is smooth and $4$-dimensional.
\end{cor}

\begin{proof}
Since $(3, -H)$ is primitive, we know that $E$ is slope-stable. Therefore, $\hom(E, E) = 1$. Moreover, we must have $\Ext^3(E, E) = \Hom(E, E(-2H))^{\vee} = 0$. By Lemma \ref{lem:no_wall_at_minus_one}, the sheaf $E$ is $\nu_{\alpha, -1}$-stable for any $\alpha > 0$. Proposition \ref{prop:max_ch2_ch3_reflexive} shows that $E(-2H)$ is reflexive, so its shift $E(-2H)[1]$ lies in the heart $\Coh^{\beta =-1}(X)$ and it is $\nu_{\alpha, -1}$-stable for any $\alpha > 0$ by Lemma \ref{lem:wall_classification_rank_three}.  
Since $\nu_{0, -1}(E) = 0 > -\tfrac{1}{2} =  \nu_{0, -1}(E(-2H)[1])$, we get $\Ext^2(E, E) = \Hom(E, E(-2H)[1]) = 0$. We can conclude $\ext^1(E, E) = \hom(E, E) - \chi(E, E) = 4$.
\end{proof}


\section{Proof of the main theorem}
\label{sec:proof}

Recall that $\overline{M}_X(v)$ is the moduli space of Gieseker-semistable sheaves with Chern character
\[
v \coloneqq \left(3, -H, -\frac{1}{2}H^2, \frac{1}{6}H^3 \right)
\]
and $M_X(v) \subset \overline{M}_X(v)$ is the open locus of Gieseker-semistable vector bundles. The aim of this section is to prove the following Theorem.

\begin{thm}
\label{thm:main}
    The moduli space $\overline{M}_X(v)$ is smooth and irreducible of dimension $4$. Moreover, there is an Abel-Jacobi morphism $\Psi \colon \overline{M}_X(v) \to J(X)$ sending $E \mapsto \widetilde{c}_2(E) - H^2$ whose image is a theta divisor $\Theta$ in the intermediate Jacobian $J(X)$. The theta divisor has a unique singular point, and $\overline{M}_X(v)$ is the blow up of $\Theta$ in this point. The exceptional divisor is isomorphic to the cubic threefold $X$ itself. 
\end{thm}

We have already shown that $\overline{M_X(v)}$ is smooth of dimension $4$ in Corollary \ref{cor:smooth_4dim}. By Proposition \ref{prop:abel-jacobi_twisted_cubic}, the image of $\varphi \colon \overline{\TT} \to J(X)$ is $\Theta \subset J(X)$, where $\TT$ is the open locus of smooth twisted cubics in the Hilbert scheme of $X$, and $\overline{\TT}$ is its closure. By Theorem \ref{thm:theta_divisor_normal}, we know that $\Theta$ is normal.

\begin{prop}
\label{prop:restriction_TT}
There is a 
surjective map $\varphi' \colon \TT \to M_X(v)$ that sends a twisted cubic $C$ to the vector bundle $E_C$. The map $\varphi|_{\TT} \colon \TT \to J(X)$ factors through $\varphi'$:  
\[
\xymatrix{
	& \TT \ar[dr]^-{\varphi|_{\TT}} \ar[dl]_-{\varphi'} \\
	M_X(v) \ar[rr]^-{\Psi|_{M_X(v)}}&&J(X) }
\]
Therefore, the image of $\Psi \colon \overline{M}_X(v) \to J(X)$ is $\Theta \subset J(X)$. 
\end{prop}

\begin{proof}
Let $C$ be a twisted cubic in $X$, then it lies in a unique hyperplane section $Y$. There is a short exact sequence    
\[
0 \to \OO_Y \to \OO_Y(C) \to T \to 0,
\]
where $T$ is a sheaf supported on $C$ with rank one. Therefore, $\widetilde{\ch}_{\leq 2}(\OO_Y(C)) = (0, H, C - H^2/2)$ and we get $\widetilde{\ch}_{\leq 2}(E_C) = (3, -H, H^2/2 - C)$. It follows that $\tilde{c}_2(E_C) = C$. Thus, the composition $\Psi|_{M_X(v)} \circ \varphi' \colon \TT \to M_X(v) \to J(X)$ is the Abel-Jacobi map $\varphi \colon \overline{\TT} \to J(X)$ restricted to $\TT$. 
Surjectivity of $\varphi'$ is a direct consequence of Theorem \ref{thm:moduli_rank_three_set_theoretic}.
\end{proof}

\begin{lem}
\label{lem:locus_is_X}
The morphism $i \colon X \to \overline{M}_X(v)$ that maps $P \mapsto K_P$ is an embedding with normal bundle $\OO_X(-H)$.
\end{lem}

\begin{proof}
We interpret $X$ as the moduli spaces of twisted ideal sheaves $\II_P(H)$ for all $P \in X$. By definition of $K_P$, we have a canonical short exact sequence
\begin{equation}
    \label{eq:K_p}
    0 \to K_P \to \OO_X^{\oplus 4} \to \II_P(H) \to 0.
\end{equation}
The appropriate version in families, considered below, induces the morphism $i$. It is injective, as $P$ is the unique point where $K_P$ is not locally free by Corollary~\ref{cor:k_p_stable}.

Applying $\Hom(\cdot, K_P)$ to \eqref{eq:K_p}, we get an isomorphism $\Ext^1(K_P, K_P) \cong \Ext^2(\II_P(H), K_P)$. Next, we apply the functor $\Hom(\II_P(H), \cdot)$ to $\eqref{eq:K_p}$ to show that the induced morphism on tangent spaces $\Ext^1(\II_P(H), \II_P(H)) \into \Ext^2(\II_P(H), K_P) = \Ext^1(K_P, K_P)$ is an embedding.
Since both $X$ and $\overline{M}_X(v)$ are smooth, the morphism is an embedding. 

To determine the normal bundle, we need a relative version of the previous arguments to determine the cokernel of this embedding as a line bundle on $X$. The universal family inducing $i$ is given by the
sheaf $\KK$ on $X \times X$ fitting into the short exact sequence
\[
0 \to \KK \to  p^*\Omega_{\P^4}|_X(H)  \to \II_\Delta(0, H) \to 0,
\]
where $p \colon X \times X \to X$ is the projection to the first factor. The pull-back of the tangent bundle via $i$ is
$i^* T_{\overline{M}_X(v)} = \HH^1(p_* \lHom(\KK, \KK))$. Since $p_* \lHom(p^*\Omega_{\P^4}|_X(H), \KK) = 0$, we have an isomorphism
\[ \HH^1(p_* \lHom(\KK, \KK)) = \HH^2(p_* \lHom(\II_\Delta(0, H), \KK). \]
The differential $d_i$ of $i$ fits into the four-term long exact sequence
\[
0 \to T_X = \HH^1(p_* \lHom(\II_\Delta(0, H), \II_\Delta(0, H))
\xrightarrow{d_i} \HH^2(p_* \lHom(\II_\Delta(0, H), \KK)) \to 
\]
\[ \to
\HH^2(p_* \lHom(\II_\Delta(0, H), p^*\Omega_{\P^4}|_X(H)))
\to \HH^2(p_* \Hom(\II_\Delta(0, H), \II_\Delta(0, H)) \to 0.
\]
Using Grothendieck duality and projection formula, the third term becomes
\[
\Omega_{\P^4}|_X(H) \otimes \HH^1(p_* \II_\Delta(0, -H) )^\vee = 
\Omega_{\P^4}|_X(H) \otimes \HH^0(p_* \OO_\Delta(0, -H) )^\vee =
\Omega_{\P^4}|_X(2H).
\]
A similar computation using the short exact sequence $I_\Delta \into \OO_X \boxtimes \OO_X \onto \OO_\Delta$ gives
\[
\HH^2(p_* \Hom(\II_\Delta, \II_\Delta) = 
\Omega_X(2H)
\]
for the 
fourth term. Thus, the cokernel of $d_i$ is isomorphic to
$\NN_{X/\P^4}^\vee(2H) = \OO_X(-H)$ as claimed.
\end{proof}

\begin{lem}
\label{lem:M_embeds}
The morphism $\Psi$ induces an isomorphism $M_X(v) \to \Theta \setminus \{0\}$. Moreover, $\Psi$ contracts the irreducible divisor $\overline{M}_X(v) \setminus M_X(v)$ to the zero point. In particular, $\Theta$ is smooth away from $0$. 
\end{lem}

\begin{proof} By Lemma~\ref{lem:E_D_vector_bundle} and Corollary~\ref{cor:k_p_stable}, the locus $\overline{M}_X(v) \setminus M_X(v)$ coincides with vector bundles $E_C$ associated to a twisted cubic $C$.
By Lemma \ref{lem:abel_jacobi_infinitesimal}, the map $\varphi|_{\TT}$ has full rank four on tangent spaces. Thus, the commutative diagram in Proposition \ref{prop:restriction_TT} implies that $\Psi|_{M_X(v)}$ has full rank four on tangent spaces. Since $M_X(v)$ is smooth of dimension four, $\Psi|_{M_X(v)}$ must be injective on tangent spaces. In particular, the morphism $\Psi|_{M_X(v)}$ must have finite fibers. Since $\phi|_{\TT}$ has generically connected fibers by Proposition \ref{prop:abel-jacobi_twisted_cubic}, the same holds for $\Psi|_{M_X(v)}$.
Since $\Theta$ is normal, Zariski's Main Theorem implies that $\Psi|_{M_X(v)}$ is an open embedding. Since $\Theta$ is singular at the origin, we must have $\Psi(M_X(v)) \subset \Theta - \{0\}$.

By definition, $\tilde{c}_2(K_P) = H^2$ and we get $\Psi(K_P) = 0$. Thus 
$\Psi^{-1}(0) = \overline{M}_X(v) \setminus M_X(v)$ and the image of $M_X(v)$ is indeed $\Theta \setminus \{0\}$ by Proposition \ref{prop:abel-jacobi_twisted_cubic}.  
\end{proof}

We can finish the proof of Theorem \ref{thm:main} with the following Lemma.

\begin{lem}
\label{lem:formal_neighborhood}
The formal neighborhood of $0 \in \Theta$ is isomorphic to the vertex of the affine cone over 
$X \subset \P^4$. Moreover, we have an isomorphism $\overline{M}_X(v) = \Bl_0(\Theta)$. Thus, $X$ is the union of all rational curves on $\overline{M}_X(v)$, and the unique divisor contracted by any morphism to a complex abelian variety.
\end{lem}

\begin{proof}
The first two claims are scheme-theoretic enhancements of the set-theoretic statements in the previous Lemma, that hold for any contraction of a divisor with ample conormal bundle to a point. We will only sketch the arguments.

Since the normal bundle of $X \subset \overline{M}_X(v)$ is anti-ample, by Artin's contractibility criterion \cite[Corollary 6.12]{Art70:contractibilityII} there is a contraction $\Psi' \colon \overline{M}_X(v) \to N$ to an algebraic space $N$ of finite type over $\C$ that is an isomorphism away from $X$, and contracts $X$ to a point $0 \in N$. Moreover, by Artin's construction in \cite[Theorem 6.2]{Art70:contractibilityII}, the formal neighborhood of $0 \in N$ is given by the affinization of the formal neighborhood of $X \subset \overline{M}_X(v)$. More precisely, if $\II$ is the ideal of $X$, then it is given by
\[ 
\Spec \varprojlim_n H^0(X, \OO_{\overline{M}_X(v)}/\II^{n+1}) = \Spec
\varprojlim_n \bigoplus_{0 \le k \le n} H^0(X, \OO_X(k)),
\]
i.e., the completion of the vertex of the affine cone over $X$. Since the image of every infinitesimal neighborhood of $X$ under $\Psi$ is affine, it factors via its affinization. Taking the limit, we see that $\Psi$ factors via $\Psi'$ both in the formal neighborhood of $X$, and in its complement. Hence (e.g. by \cite[Theorem 3.1]{Art70:contractibilityII}) we get an induced morphism $j \colon N \to \Theta$ factoring
$\Psi$. As $j$ is bijective on points and with normal target, it is an isomorphism.

For the last claim, note that $X$ is uniruled, hence the union $U$ of all rational curves in $\overline{M}_X(v)$ contains $X$. If there was any other rational curve $C$ not contained in $X$, then $\Psi \colon C \to \Theta$ is a non-constant map from a rational to an abelian variety, a contradiction.
\end{proof}

\begin{cor} \label{cor:classicalTorelli}
 If $X_1$ and $X_2$ are smooth projective threefolds with $J(X_1) = J(X_2)$ as principally polarised abelian varieties, then $X_1 = X_2$.
 \end{cor}
\begin{proof}
 As in the classical argument, this is an immediate consequence of the description of the singularity of the theta divisor in Lemma \ref{lem:formal_neighborhood}.
\end{proof}

\section{Kuznetsov component}
\label{sec:kuznetsov}

The bounded derived category of a cubic threefold $X$ admits a semi-orthogonal decomposition 
\begin{equation*}
\Db(X) = \langle \Ku(X), \OO_X, \OO_X(1) \rangle
\end{equation*}
whose non-trivial part $\Ku(X)$ is called the Kuznetsov component. The goal of this section is to give a new proof of the following Theorem.

\begin{thm}
\label{thm:categorical-Torelli}
Let $X_1$ and $X_2$ be smooth cubic threefolds. Then $\Ku(X_1)$ and $\Ku(X_2)$ are equivalent as triangulated categories if and only if $X_1$ and $X_2$ are isomorphic.
\end{thm}


Let $S$ be the Serre functor of $\Ku(X)$. By \cite[Lemma 4.1 and Lemma 4.2]{Kuz04:SOD}, for any object $F \in \Ku(X)$, we have 
\begin{equation} \label{eq:SerrefunctorKuX}
S(F) = L_{\OO_X}(F \otimes \OO_X(H))[1]
\end{equation}
where $L_{\OO_X}$ is the left mutation functor with respect to $\OO_X$.  
By \cite[Proposition 2.7]{BMMS12:cubics}, the numerical Grothendieck group $\NN(\Ku(X))$ is a two-dimensional lattice 
\[
    \NN(\Ku(X)) \cong \mathbb{Z}^2 \cong \Z[\II_{\ell}] \oplus \Z[S(\II_{\ell})]
\]
where $\II_{\ell}$ is the ideal sheaf of a line $\ell$ in $X$. With respect to this basis, the Euler characteristic $\chi(-, -)$ on $\mathcal{N}(\Ku(X))$ has the form
\[
  \begin{bmatrix}
    -1 & -1 \\
    0 & -1
  \end{bmatrix}. 
\]
For any line $\ell$ in $X$, we know $\ch(\II_{\ell}) = (1, 0, -\frac{1}{3}H^2, 0)$. The Chern character of our second basis vector of $\ch(\Ku(X))$, and the action of the Serre functor $S$ on our chosen basis are given as follows.

\begin{lem}
\label{lem:chern-characters}
We have 
$\ch(S(I_{\ell})) = (2, -H, -\tfrac{1}{6} H^2, \tfrac{1}{6} H^3)$ and $\ch(S^2(I_{\ell})) = (1, -H, \tfrac{1}{6} H^2, \tfrac{1}{6} H^3)$. Thus, the class $[S^2(\II_{\ell})]$ in $\NN(\Ku(X))$ is equal to $[S(\II_{\ell})] - [\II_{\ell}]$.
\end{lem}
\begin{proof}
By \eqref{eq:SerrefunctorKuX} we have $[S(E)] = -[E(H)] + \chi(E(H))[\OO_X]$ for $E \in \Ku(X)$.  Hence $\ch(I_\ell(H)) = (1, H, \frac 16 H^2, -\frac 16 H^3)$ and $\chi(I_\ell(H)) = 3$ imply the formula for $\ch(S(I_\ell))$. The formula for $\ch(S^2(I_\ell))$ follows from the last claim, which in turn follows from the Euler characteristic form above with 
\begin{align*}
\chi(I_\ell, S^2(I_\ell)) &= \chi(S^2(I_\ell), S(I_\ell)) =  \chi(S(I_\ell), I_\ell) = 0  = \chi([I_\ell], [S(\II_{\ell})] - [\II_{\ell}]) \quad \text{and} \\
\chi(S(I_\ell), S^2(I_\ell)) &= \chi(I_\ell, S(I_\ell)) = -1 = \chi([S(I_\ell)], [S(\II_{\ell})] - [\II_{\ell}]). \qedhere
\end{align*}
\end{proof}

For a point $P \in X$, the sheaf $K_P$ which is defined through the sequence \eqref{k-p}, lies in the Kuznetsov component $\Ku(X)$.  
\begin{lem}
\label{lem:clases-3}
Let $[A]$ be a class in $\NN(\Ku(X))$ such that $\chi([A], [A]) = -3$. Then, up to a sign, $[A]$ is either $[K_P] = [\II_{\ell}] +[S(\II_{\ell})]$, or $[S(K_P)] = -[\II_{\ell}] +2[S(\II_{\ell})]$, or $[S^2(K_p)] = -2[\II_{\ell}] +[S(\II_{\ell})]$.
\end{lem}
Let $\sigma^0_{\alpha, -1/2} = \big(\Coh^0_{\alpha, -1/2}(X), Z^0_{\alpha, -1/2}\big)$ be the weak stability condition on $\Db(X)$ constructed in \cite[Proposition 2.14]{BLMS17:stability_kuznetsov}. Here $\Coh^0_{\alpha, -1/2}(X)$ is the usual double tilt and 
\begin{equation}
\label{tilt-stability_function}
    Z^0_{\alpha, -\frac{1}{2}}(E) = H^2 \cdot \ch_1^{-\frac{1}{2}}(E) + i \left( H \cdot \ch_2^{-\frac{1}{2}}(E) - \frac{\alpha^2}{2} H^3 \cdot \ch_0(E)\right). 
\end{equation}
As proven in \cite[Theorem 6.8]{BLMS17:stability_kuznetsov}, for $0 < \alpha \ll 1$, it induces the stability condition $\sigma(\alpha) = \left(\AA(\alpha), Z(\alpha) \right)$ on $\Ku(X)$ where  
\[
\AA(\alpha) \coloneqq \Coh^0_{\alpha, -\frac{1}{2}}(X) \cap \Ku(X) \qquad \text{and} \qquad Z(\alpha) \coloneqq Z^0_{\alpha, -\frac{1}{2}}|_{\Ku(X)}. 
\]

\begin{lem}
\label{lem:embedding-Gieseker}
There is an embedding $M_X(v) \into M_{\sigma(\alpha)}\big([\II_{\ell}] +[S(\II_{\ell})]\big)$ from the moduli space $M_{X}(v)$ for $v = \ch(\II_{\ell}) + \ch(S(\II_{\ell})) = (3, -H, -H^2/2, H^3/6)$ to $M_{\sigma(\alpha)}\big([\II_{\ell}] +[S(\II_{\ell})]\big)$ which parameterises $\sigma(\alpha)$-semistable objects in $\Ku(X)$ of class  $[\II_{\ell}] +[S(\II_{\ell})] \in \mathcal{N}(\Ku(X))$.
\end{lem}

\begin{proof}
According to Lemma \ref{lem:no_wall_at_minus_one} there is no wall for objects of Chern character $v$ to the left of the vertical wall. Thus, $E$ is $\nu_{\alpha, -1/2}$-stable for any $\alpha >0$. Since $\sigma^0_{\alpha, -1/2}$ is just a rotation of $\nu_{\alpha, -1/2}$, we obtain that $E$ is $\sigma^0_{\alpha, -1/2}$-stable. By Theorem \ref{thm:moduli_rank_three_set_theoretic}, part (ii), the sheaf $E \in \Ku(X)$ lies in the Kuznetsov component. Thus, $E$ is $\sigma(\alpha)$-stable. Note that the object $E$ could be destabilised by objects with $Z^0_{\alpha, -\frac{1}{2}} = 0$ after rotation. But we know these are all sheaves supported in dimension zero and would not be in $\Ku(X)$ and therefore, $E$ is stable after restriction to $\Ku(X)$.   
\end{proof}

\cite[Corollary 5.6]{PY20:Fano3index2} implies that the stability condition $\sigma(\alpha)$ is $S$-invariant, i.e., $S \cdot \sigma(\alpha) = \sigma(\alpha) \cdot \tilde{g}$ for $\tilde{g} \in \widetilde{\text{GL}}^+(2, \mathbb{R})$.   
Thus, there is an isomorphism 
\begin{align}
  	S \colon  M_{\sigma(\alpha)}\big( 2[I_{\ell}] -[S(\II_{\ell})] \big)& \rightarrow M_{\sigma(\alpha)}\big( [\II_{\ell}] +[S(\II_{\ell})] \big)\nonumber \\
  	 E& \mapsto S(E) \label{Serre functor-moduli}
\end{align}

The following Proposition is a slight strengthening of \cite[Theorem 1.2]{APR19:kuznetsov_fano_torsion_moduli} which describes all elements of the moduli space. The idea of the proof is the same as \cite[Lemma 2.2]{APR19:kuznetsov_fano_torsion_moduli}.

\begin{prop}
Any $\sigma(\alpha)$-semistable object in $\Ku(X)$ of class $2[\II_{\ell}] - [S(\II_{\ell})]$ is of the form $G[2k]$ for $k \in \Z$ where $G$ is either equal to $G_P(-H)$ described in \eqref{G-p} for a point $P \in X$, or $\OO_Y(D - H)$ where $D$ is a Weil-divisor on some $Y \in |H|$.
\end{prop}

\begin{proof}
Lemma \ref{lem:chern-characters} implies $\ch(G) = (0, H, -\tfrac{1}{2} H^2 , -\tfrac{1}{6} H^3)$. Since $G$ is $\sigma(\alpha)$-semistable, its shift $G[2k]$ lies in the heart $\AA(\alpha)$ for some $k \in \mathbb{Z}$. We know its image under the stability function $Z(\alpha)$ is equal to $-H^3$, so it has maximum phase in the heart $\AA(\alpha)$ which immediately implies $G[2k]$ is $\sigma^0_{\alpha, -1/2}$-semistable. We claim that $G[2k]$ has no subobject $Q \in \Coh^0_{\alpha, -1/2}$ with $Z^0_{\alpha, -1/2}(Q) = 0$, so it is $\nu_{\alpha, -1/2}$-semistable. Assume for a contradiction that there is such a subobject $Q$. By the definition of $\Coh^0_{\alpha, -1/2}(X)$, it is a sheaf supported in dimension zero. Thus, $\hom(\OO_X , Q) \neq 0$. Since $\OO_X \in \Coh^0_{\alpha, -1/2}(X)$, we have $\hom(\OO_X, (G[2k]/Q)[-1]) = 0$. Therefore, $\hom(\OO_X, G[2k]) \neq 0$ which is not possible because $G[2k] \in \Ku(X)$. Finally, since $G[2k]$ is $\nu_{\alpha, -1/2}$-semistable for $0 < \alpha \ll 1$, the claim follows by Proposition \ref{prop:torsion_classification}, (ii).  
\end{proof}

\begin{rmk}
Since the class $2[\II_{\ell}] -[S(\II_{\ell})]$ is primitive in $\NN(\Ku(X))$, any $\sigma(\alpha)$-semistable object of this class is $\sigma(\alpha)$-stable if we choose $\alpha$ sufficiently small.
\end{rmk}

We now describe the image of the semistable objects $G \in M_{\sigma(\alpha)}\big(2[\II_{\ell}] -[S(\II_{\ell})]\big)$ under the Serre functor $S$. If $G = G_P(-H)$, then by \eqref{G-p}, we know there is a distinguished triangle
\[
    \OO_X[1] \to G_P \to \II_P(H) \to \OO_X[2]
\]
which gives $L_{\OO_X}(G_P) = L_{\OO_X}(\II_P(H)) = K_P[1]$, so 
\begin{equation}
\label{serre-1}
    S(G_P) = K_P[2]. 
\end{equation}
If $G = \OO_Y(D -H)$, then $G(H) = \OO_Y(D)$ is of class $\left(0, H, \frac{1}{2}H^2, -\tfrac{1}{6} H^3 \right)$, and lies in a distinguished triangle 
\[
    \OO_X^{\oplus 3} \to \OO_Y(D) \to E_D[1] \to \OO_X^{\oplus 3}[1]. 
\]
Thus,
\begin{equation}
\label{serre-2}
    S(G) = L_{\OO_X}(\OO_Y(D))[1] = L_{\OO_X}\big(E_D[1]\big)[1] = E_D[2]. 
\end{equation} 
 Combining \eqref{serre-1} and \eqref{serre-2} with Lemma \ref{lem:embedding-Gieseker} implies the next result. 

\begin{thm}
\label{thm:equalmodspace}
The moduli space $M_{\sigma(\alpha)}\big([\II_{\ell}] +[S(\II_{\ell})] \big)$ is isomorphic to the moduli space $\overline{M}_X(v)$ parametrising Gieseker-stable sheaves of class $v$.
\end{thm}

The next step is to show that we can replace $\sigma(\alpha)$ by any $S$-invariant stability condition on $\Ku(X)$.

\begin{lem}\cite[Lemma 5.8 and 5.10]{PY20:Fano3index2}\label{lem:PY-1}
Let $\sigma$ be an $S$-invariant stability condition on $\Ku(X)$ and $F \in \Ku(X)$ be $\sigma$-semistable of phase $\phi(F)$. Then 
\begin{enumerate}
    \item $\phi(F) < \phi(S(F)) < \phi(F)+2$. 
    \item $\dim \Ext^1(F, F) \geq 2$. 
\end{enumerate}
\end{lem}
For cubic threefolds, we also have a weak version of Mukai Lemma for K3 surfaces. 
\begin{lem}(Weak Mukai Lemma)\ \cite[Lemma 5.11]{PY20:Fano3index2}\label{lem:PY-2}
Let $\sigma$ be an $S$-invariant stability condition. Let $A \to E \to B$ be a triangle in $\Ku(X)$ such that $\hom(A, B) = 0$ and the $\sigma$-semistable factors of $A$ have phase greater than or equal to the phase of the $\sigma$-semistable factors of $B$. Then
\[
\dim_{\mathbb{C}} \Ext^1(A, A) + \dim_{\mathbb{C}} \Ext^1(B,B) \leq \dim_{\mathbb{C}} \Ext^1(E, E).
\]
\end{lem}

\begin{prop}
\label{prop:unique}
Let $\sigma_1$ and $\sigma_2$ be two $S$-invariant stability conditions on $\Ku(X)$. An object $E \in \Ku(X)$ of class $[\II_{\ell}]+[S(\II_{\ell})]$ is $\sigma_1$-stable if and only if it is $\sigma_2$-stale.  
\end{prop}

\begin{proof}
By \cite[Proposition 4.6]{PY20:Fano3index2}, $\II_{\ell}$ and $S(\II_{\ell})$ are $\sigma$-stable with respect to any $S$-invariant stability condition. Thus, Lemma \ref{lem:PY-1} implies that 
\begin{equation}
\label{phase order}
\phi_{\sigma}(\II_{\ell})  < \phi_{\sigma}(S(\II_{\ell})) < \phi_{\sigma}(\II_{\ell}) +2. 
\end{equation}
Take a $\sigma_1$-stable object $E \in \Ku(X)$ of class $[\II_{\ell}]+[S(\II_{\ell})]$. Since $\sigma_1$ is $S$-invariant, Lemma \ref{lem:PY-1} gives 
\[
    \phi_{\sigma_1}(E)  < \phi_{\sigma_1}(S(E)) < \phi_{\sigma_1}(E) + 2. 
\]
Thus, for $i < 0$ or $i \geq 2$, we get  
\[
    \hom(E, E[i]) = \hom(E[i], S(E)) = 0. 
\]
Since $E$ is $\sigma_1$-stable, we get $\hom(E, E) = 1$ which gives 
\[
    \hom(E, E[1]) = -\chi(E, E) +1 = 4. 
\]
Suppose now for a contradiction that $E$ is $\sigma_2$-unstable. There is a distinguished triangle of destabilising objects $F_1 \to E \to F_2 \to F_1[1]$ with respect to $\sigma_2$. We may assume $F_1$ is $\sigma_2$-semistable. Thus, Lemma \ref{lem:PY-1} implies that 
\begin{equation}
\label{c.1}
    \hom(F_1,F_1[1]) \geq 2.
\end{equation}
Since the phase of $F_1$ is bigger than the phase of $\sigma_2$-semistable factors of $F_2$, we have 
\begin{equation}
    \hom(F_1, F_2) = 0. 
\end{equation}
Thus, Weak Mukai Lemma \ref{lem:PY-2} implies
\[
    \hom(F_1, F_1[1]) + \hom(F_2, F_2[1]) \leq \hom(E, E[1]) = 4.
\]
By \eqref{c.1}, we get $\hom(F_2, F_2[1]) \leq 2$. If $\hom(F_2, F_2[1]) = 0$ or $1$, then all its $\sigma_2$-semistable factors would satisfy the same property by Weak Mukai Lemma \ref{lem:PY-2} which is not possible by Lemma \ref{lem:PY-1}. Therefore,
\[
    \hom(F_1, F_1[2]) = \hom(F_2, F_2[1]) = 2
\]
and \cite[Lemma 5.12]{PY20:Fano3index2} implies that $F_1$ and $F_2$ are $\sigma_2$-stable. This gives $\chi(F_i, F_i) = -1$ for $i=1, 2$, so $[F_i]$ is either $\pm [\II_{\ell}]$, or $\pm [S(\II_{\ell})]$, or $\pm([S(\II_{\ell})] - [\II_{\ell}])$. Since there are only 2 stable factors and the object $E$ is of class $[\II_{\ell}] + [S(\II_{\ell})]$, the destabilising objects must be of class $[\II_{\ell}]$ and $[S(\II_{\ell})]$. Thus, \cite[Proposition 4.6]{PY20:Fano3index2} implies that the destabilising objects are $\II_{\ell}[2k]$ and $S(\II_{\ell'})[2k']$ for two lines $\ell, \ell'$ and integers $k, k' \in \mathbb{Z}$. 

Let $F_1 = \II_{\ell}[2k]$ and $F_2 = S(\II_{\ell'})[2k']$. 
Since $E$ is $\sigma_1$-stable, we have $\phi_{\sigma_1}(F_1) < \phi_{\sigma_1}(F_2)$, thus \eqref{phase order} gives $k \leq k'$. But $F_1$ and $F_2$ are the destabilising objects with respect to $\sigma_2$, hence $\phi_{\sigma_2}(F_1) > \phi_{\sigma_2}(F_2)$ and \eqref{phase order} gives $k'+1 \leq k$ which is not possible. By a similar argument, we reach a contradiction if $F_1 = S(\II_{\ell'})[2k']$ and $F_2 = \II_{\ell}[2k]$. Finally, note that $E$ cannot be strictly $\sigma_2$-semistable because the phases of $\II_{\ell}[2k]$ and $S(\II_{\ell})[2k']$ cannot be equal by \eqref{phase order}. 
\end{proof}

\begin{proof}[Proof of Theorem \ref{thm:categorical-Torelli}]
As a cubic threefold has free Picard group of rank one, the first implication is obvious. As for the second implication, assume there is an exact equivalence $\Phi \colon \Ku(X_1) \rightarrow \Ku(X_2)$. Lemma \ref{lem:clases-3} implies that, up to composing with a power of the Serre functor of $\Ku(X_1)$ and shift functor, we may assume $[\Phi_*(K_P)] = [K_{P'}]$ for points $P, P'$ in $X_1$ and $X_2$, respectively. Take an $S$-invariant stability condition $\sigma$ on $\Ku(X_1)$. Theorem \ref{thm:equalmodspace} and \ref{prop:unique} imply that 
\begin{equation}\label{eq:iso}
M_{X_1}(v) \cong M_{\sigma}\big(\Ku(X_1) ,  [K_P]\big) \cong M_{\phi \cdot \sigma}\big(\Ku(X_2), [K_{P'}]\big).
\end{equation}
Since the Serre functor commutes with auto-equivalences, $\phi \cdot \sigma$ is an $S$-invariant stability condition on $\Ku(X_2)$. Thus, Theorem \ref{thm:equalmodspace} gives
\[
    M_{\phi \cdot \sigma}\big( \Ku(X_2), [K_{P'}]\big) \cong M_{X_2}(v). 
\]
Combining this with \eqref{eq:iso} gives $M_{X_1}(v) \cong M_{X_2}(v)$. By Lemma \ref{lem:formal_neighborhood}, we know $X_1$ and $X_2$
are the unique exceptional divisors of $M_{X_1}(v)$ and $M_{X_2}(v)$ which get contracted by any map to a complex abelian variety. Thus, $X_1 \cong X_2$.  
\end{proof}

\end{document}